\documentclass[11pt,reqno, a4paper]{amsart}

\usepackage{amsmath}
\usepackage{amssymb}
\usepackage{amscd}
\usepackage{amsfonts}
\usepackage{mathrsfs}
\usepackage{dsfont}
\usepackage{setspace}
\usepackage{version}
\usepackage{mathtools}
\usepackage{soul,tikz}
\usepackage{tikz-cd}

\usepackage{enumitem}
\usepackage{pgfplots,wrapfig}

\usepackage{xcolor}

\definecolor{crimson}{rgb}{0.85, 0.08, 0.4}
\definecolor{bleudefrance}{rgb}{0.2, 0.5, 0.9}
\usepackage[pdftex,colorlinks,linkcolor=crimson,citecolor=bleudefrance,filecolor=black]{hyperref}
\usepackage[utf8]{inputenc}

\newtheorem{theorem}{Theorem}[section]
\newtheorem{lemma}[theorem]{Lemma}
\newtheorem{corollary}[theorem]{Corollary}
\newtheorem{proposition}[theorem]{Proposition}

\theoremstyle{definition}

\newtheorem{example}[theorem]{Example}

\theoremstyle{remark}
\newtheorem{remark}[theorem]{Remark}

\numberwithin{equation}{section}

\renewcommand{\leq}{\leqslant}
\renewcommand{\geq}{\geqslant}
\newsavebox{\proofbox}
\savebox{\proofbox}{\begin{picture}(7,7)  \put(0,0){\framebox(7,7){}}\end{picture}}

\newcommand\E{\mathbb{E}}

\newcommand\R{\mathbb{R}}

\newcommand\C{\mathbb{C}}
\newcommand\N{\mathbb{N}}

\newcommand\HH{\mathbb{H}}
\newcommand\SL{\operatorname{SL}}

\newcommand\Aff{\operatorname{Aff}}

\newcommand\im{\operatorname{Im}}

\newcommand\supp{\operatorname{supp}}
\newcommand\GL{\operatorname{GL}}

\newcommand\PGL{\operatorname{PGL}}
\newcommand\SO{\operatorname{SO}}

\newcommand\Sp{\operatorname{Sp}}

\newcommand\Endo{\operatorname{End}}

\renewcommand\P{\mathbb{P}}

\renewcommand\b{{\bf b}}
\newcommand\acts{\operatorname{\curvearrowright}}

\def\a{\mathfrak{a}}

\def\b{\mathfrak{b}}

\newcommand\id{\operatorname{id}}

\newcommand{\efface}[1]{}

\begin{document}

\author{Richard Aoun}
\address{University Gustave Eiffel, Champs-sur-Marne, 
5 boulevard Descartes Champs-sur-Marne 77454 Marne-la-Vallée Cedex 2, France}
\email{richard.aoun@univ-eiffel.fr}
\thanks{R.A. is  supported  by a Research Group Linkage Programme from the Humboldt Foundation}

\author{Cagri Sert}
\address{Institut f\"{u}r Mathematik, Universit\"{a}t Z\"{u}rich, 190, Winterthurerstrasse, 8057 Z\"{u}rich, Switzerland}
\email{cagri.sert@math.uzh.ch}
\thanks{C.S. is supported by SNF Ambizione grant 193481}

\subjclass[2010]{Primary 37H15; Secondary 60J05, 60B15, 37A20}
\keywords{Stationary measures, projective space, Lyapunov exponents, random matrix products, fibered dynamical systems}

\title[Stationary measures on projective spaces]{Stationary probability measures on projective spaces 1:  block-Lyapunov dominated systems}

\begin{abstract}
Given a finite-dimensional real vector space $V$, a probability measure $\mu$ on $\PGL(V)$ and a $\mu$-invariant subspace $W$, under a block-Lyapunov contraction assumption, we prove existence and uniqueness of lifts to $P(V)\setminus P(W)$ of stationary probability measures on the quotient $P(V/W)$. In the other direction, i.e.~ block-Lyapunov expansion, we prove that stationary measures on $P(V/W)$ have lifts if any only if the group generated by the support of $\mu$ stabilizes a subspace $W'$ not contained in $W$ and exhibiting a faster growth than on $W \cap W'$. These refine the description of stationary probability measures on projective spaces as given by Furstenberg, Kifer and Hennion, and under the same assumptions, extend corresponding results by Aoun, Benoist, Bru\`{e}re, Guivarc'h, and others.
\end{abstract}

\setcounter{tocdepth}{1}
\maketitle

\section{Introduction}

Let $V$ be a finite dimensional real vector space and $\mu$ a probability measure on $\GL(V)$. Let $X_1,X_2,\ldots$ denote $\GL(V)$-valued iid random variables with distribution $\mu$ and write $L_n=X_n \ldots X_1$ for the associated random matrix product. Via the action of $\GL(V)$ on the projective space $P(V)$ the random product $L_n$ induces a Markov chain on $P(V)$. The stationary measures (called $\mu$-stationary measure) on $P(V)$ encode to a great extent the asymptotic behaviour of the random product $L_n$. This is for instance manifested by their ubiquitous role in the limit theorems for random matrix products starting with the work of Furstenberg--Kesten \cite{furstenberg-kesten}. Recall that the latter work proved the non-commutative extension of the law of large numbers which says that if $\mu$ has a finite first moment (i.e.~ $\int \log N(g) d\mu(g)<\infty$, where $N(g)=\max\{\|g\|,\|g^{-1}\|\}$ for a choice of norm on $V$), then there exists a constant $\lambda_1$ such that $\mu^\N$-a.s.
\begin{equation}\label{eq.lambda1.intro}
\frac{1}{n} \log \|L_n\| \underset{n \to \infty}{\longrightarrow} \lambda_1.
\end{equation}
The constant $\lambda_1(\mu)$ is called the top Lyapunov exponent and it is expressed via the Furstenberg formula as an integral with respect to a $\mu$-stationary measure $\nu$ on $P(V)$: 
\begin{equation}\label{eq.furstenberg.formula.1.intro}
\lambda_1(\mu)=\iint \log \frac{\|gv\|}{\|v\|} d\mu(g) d\nu(\R v).
\end{equation}
The subsequent work of Furstenberg \cite{furstenberg.non.commuting,Furstenberg.boundary} clarified qualitatively the description of $\mu$-stationary measures on $P(V)$ under irreducibility assumptions; the quantitative study of these measures (their dimensions, regularity properties etc.) are current topics of study. Here the probability measure $\mu$ is called (resp.~ strongly) irreducible if the semigroup $\Gamma_\mu$ generated by its support does not preserve a proper non-trivial subspace of $V$ (resp.~ the union of finite collection of such spaces). Otherwise, we shall say that $\mu$ is reducible. 

The theory of random matrix products encompasses many interesting situations when the irreducibility assumption is dropped; for example the study of random difference equations or affine recursion $X_{n+1}=A_n X_n +B_n$ ($A_n \in \GL_d(\R)$, $X_n,B_n \in \R^d$) belongs to this setting --- see more examples in \S \ref{sec.ex.and.app}. On the other hand, except in some particular situations (mainly the affine recursion which is extensively studied \cite{babillot-bougerol-elie,bougerol-picard,brandt, diaconis-freedman,hutchinson,kesten.random.difference} and more recently \cite{brofferio-buraczewski-damek, brofferio-peigne-pham, buraczewski.book}), the theory of random matrix products -- in particular, the description of stationary measures on projective spaces -- is much less complete without the irreducibility assumptions. The remarkable works due to Furstenberg--Kifer \cite{furstenberg-kifer} and Hennion \cite{hennion} established a description of stationary measures without any irreducibility assumption (implying, for example, continuity of Lyapunov exponents with respect to transition probabilities, see e.g.~ \cite{peres.oberwolfach}); however many natural questions remain open. 

The more recent works of Benoist--Bru\`{e}re \cite{benoist-bruere} and Aoun--Guivarc'h \cite{aoun-guivarch} refine the works of Furstenberg--Kifer and Hennion under various additional assumptions by giving a more precise description of stationary measures. These assumptions, among others, include certain \textit{domination conditions} imposing different speeds (Lyapunov exponents) in different subspaces preserved by $\Gamma_\mu$ -- excepting the treatment of a non-dominated case in \cite{benoist-bruere} under other algebraic assumption. 

Our goal in this paper is to give a description of stationary measures on $P(V)$ \textit{only under domination conditions}. In other words, in previous works, we eliminate all other hypotheses not pertaining to the domination of Lyapunov exponents. For example, our framework eliminates all algebraic assumptions in \cite{benoist-bruere, bougerol-picard} and considerably relaxes the domination assumptions in \cite{aoun-guivarch, brandt}. Our work in the contracting case below can also been as following up the well-studied field of contracting-on-average iterated function system (IFS), we comment on more on this in \S \ref{subsub.history.intro}. Let us now continue by introducing the necessary notions to express our framework and to state our results precisely.\\

\vspace{0.5cm}
Let $G$ be a measurable group acting measurably on a measure space $X$. Let $\mu$ be a probability measure on $G$. A probability measure $\nu$ on $X$ is said to be $\mu$-\textit{stationary} if $\nu=\mu \ast \nu$. Here, $\mu \ast \nu$ is a probability on $X$ for which the integral of a bounded measurable function $f$ is given by $
\iint f (gx) d\mu(g) d \nu(x).$
In this article, $V$ will denote a finite dimensional real or complex vector space and we will mostly be concerned with closed subgroups of $\GL(V)$ acting on $P(V)$. 

All probability measures considered in this paper will be supposed to have a finite first moment. The top (or first) Lyapunov exponent of a probability measure $\mu$ is defined as in \eqref{eq.lambda1.intro}. The other exponents $\lambda_i(\mu)$ for $i=1,\ldots,d=\dim V$ are defined by stipulating that $\lambda_1(\mu)+\ldots+ \lambda_k(\mu)$ is the almost sure limit 
$\lim_{n \to \infty} \frac{1}{n} \log \|\bigwedge^k L_n\|$, where $\bigwedge^k L_n$ denoted the $k^{th}$-exterior power of $L_n$. Clearly, $\lambda_1(\mu) \geq \lambda_2(\mu) \geq \cdots \geq \lambda_d(\mu)$. Whenever a probability measure $\mu$ is understood and $W<V$ is a $ \Gamma_\mu$-invariant subspace (we equivalently say $W$ is $\mu$-invariant), $\mu$ induces a measure on $\GL(W)$ and $\GL(V/W)$, we will denote the associated Lyapunov exponents, respectively, by $\lambda_i(W)$ and $\lambda_i(V/W)$. 

Finally, given a $\mu$-stationary probability measure $\nu$ on $P(V)$, its \textit{cocycle average}\footnote{The map $(g,\R v) \mapsto \log \frac{\|gv\|}{\|v\|}$ satisfies the additive cocycle property, whence comes the terminology.} is the quantity $\alpha(\nu)$ defined as
\begin{equation}\label{eq.defn.cocycle.av.intro}
\alpha(\nu)= \iint \log \frac{\|gv\|}{\|v\|} d\mu(g) d\nu(\R v).
\end{equation}
The definition does not depend on the choice of the norm $\|.\|$ and when $\nu$ varies over various $\mu$-stationary ergodic probability measures $\nu$, the values taken by $\alpha(\nu)$ is contained (possibly properly) in the set of Lyapunov exponents $\{\lambda_1(\mu),\lambda_2(\mu),\ldots,\lambda_d(\mu)\}$, always attaining the top Lyapunov exponent (see \cite{furstenberg-kifer,hennion}, more will be discussed below, and in detail in \S \ref{sec.prelim}). Here, ergodic means that $\nu$ is extremal in the convex set of $\mu$-stationary probability measures.

\subsection{Contracting case}
We are now ready to state our first result.

\begin{theorem}[Contracting case: existence and uniqueness of lifts]\label{thm.lift.exist.unique}
Let $\mu$ be a probability measure on $\GL(V)$ with finite first moment and  $W$ a $\mu$-invariant subspace. Then for every $\mu$-stationary ergodic probability measure $\overline{\nu}$ on $P(V/W)$ whose cocycle average satisfies $$\alpha(\overline{\nu}) >\lambda_1(W),$$ there exists a unique $\mu$-stationary lift $\nu$ on $P(V)\setminus P(W)$. 
\end{theorem}

Here and elsewhere, we employ the term \textit{lift} to mean that $\nu$ is a $\mu$-stationary probability measure on $P(V) \setminus P(W)$ whose push-forward on $P(V/W)$ under the map induced by the natural projection $V \to V/W$ is $\overline{\nu}$.

Before proceeding with associated equidistribution result, let us briefly comment on the particular cases of this result in the literature. The case where $W$ is a hyperplane in $V$, i.e.~ the base $P(V/W)$ is a singleton, corresponds to affine stochastic recursion, i.e.~ affine random walks. In this case, the unique stationary measure on $P(V/W)$ is trivial and the associated cocycle average is zero (i.e.~ $\alpha(\overline{\nu})=0$). In his influential paper, Hutchinson \cite{hutchinson} proved the existence and uniqueness of stationary measure under strict contraction hypothesis. The contracting-on-average version of Hutchinson's result (i.e.~ when $0>\lambda_1(W)$) was proved by Vervaat \cite{vervaat}, Brandt \cite{brandt} and Bougerol--Picard \cite{bougerol-picard}. 
More recently, Benoist--Bru\`{e}re \cite{benoist-bruere} proved the above result when $\mu$ has bounded support and $V/W$ as well as $W$ are supposed to be strongly irreducible and proximal. Finally, Aoun--Guivarc'h \cite{aoun-guivarch} treated the case when $W$ is the second Furstenberg--Kifer--Hennion space (see below) and $\lambda_1(\mu)>\lambda_2(\mu)$, in particular when $V/W$ is strongly irreducible and proximal. The previous result therefore generalizes the corresponding results in these works.
 
\bigskip

The following is an associated equidistribution result which is, in fact, a crucial ingredient for Theorem \ref{thm.lift.exist.unique}.

\begin{proposition}[Equidistribution]
Keep the setting of Theorem \ref{thm.lift.exist.unique}.
\begin{enumerate}
\item[(1)]  For any $x\in P(V)\setminus P(W)$,  the Ces\`{a}ro means $\frac{1}{n}\sum_{i=1}^n{\mu^{\ast i} \ast \delta_{x}}$ converge weakly to $\nu$  if and only if $\frac{1}{n}\sum_{i=1}^n{\mu^{\ast i} \ast \delta_{\overline{x}}}$ converges weakly to $\overline{\nu}$, where $\overline{x}$ denotes the projection of $x$ to $P(V/W)$.
\item[(2)] For any $x \in P(V)\setminus P(W)$, the sequence $\frac{1}{n}\sum_{i=1}^n \delta_{X_i \cdots X_1 \cdot x}$ converges to $\nu$ a.s. if any only if $\frac{1}{n}\sum_{i=1}^n \delta_{X_i \cdots X_1 \cdot \overline{x}}$ converges to $\overline{\nu}$ a.s.
\end{enumerate}
\label{prop.equidistribution.contracting}
\end{proposition}

\subsubsection{Some consequences}
We proceed to single out some consequences of the previous results. The first consequence will provide a refinement of the description of stationary measures given by Furstenberg--Kifer \cite{furstenberg-kifer} and Hennion \cite{hennion}. To state the consequences, we need to introduce some terminology (coming from \cite{furstenberg-kifer, hennion}).

Given a probability measure $\mu$ on $\GL(V)$, let $\mathcal{M}_\mu(V)$ denote the set of $\mu$-stationary and ergodic probability measures on $P(V)$. As $\nu$ varies in the set $\mathcal{M}_\mu(V)$, the values taken by the cocycle average $\alpha(\nu)$ describe the set of Furstenberg--Kifer--Hennion exponents (for short, we will say FKH exponents) $\infty>\beta_1(\mu)>\beta_2(\mu)>\ldots >\beta_k(\mu)=:\beta_{\min}(\mu)>-\infty$. For each FKH exponent $\beta_i(\mu)$, there exists a $\mu$-invariant subspace, denoted $F_i(\mu)$ and called $i^{th}$ FKH space, maximal for the property that that the top Lyapunov exponent of $\mu$ on $F_i(\mu)$ is $\beta_i(\mu)$. We have the FKH filtration given by $V=F_1(\mu) \supsetneq F_2(\mu) \cdots \supsetneq F_k(\mu) \supsetneq \{0\}$. By convention, we set $F_{k+1}(\mu)=\{0\}$ and $\beta_{k+1}(\mu)=-\infty$. When $\mu$ is understood, we will omit the dependence on $\mu$ from the notation and in this context the index $k$ will be the index of the smallest real FKH exponent. We warn the reader that as opposed to Oseledets' theorem, the FKH spaces depend only on $\mu$, in particular, they are not random. The following is a direct consequence of Theorem \ref{thm.lift.exist.unique} and results of \cite{furstenberg-kifer,hennion}.

\begin{corollary}\label{corol.sur.Fi/Fi+1}
Let $\mu$ be a probability measure on $\GL(V)$ with finite first moment. Then, for every $\mu$-stationary ergodic probability measure $\nu$ on $P(V)$, there exists $i=i_\nu\in \{1,\cdots, k\}$ such that $\nu$ is the unique lift on $P(F_{i_\nu})\setminus P(F_{i_\nu+1})$ of an ergodic $\mu$-stationary probability measure on the quotient space $F_{i_\nu}/F_{i_\nu+1}$.
\end{corollary}

The index $i_\nu \in \{1,\ldots,k\}$ appearing above is uniquely defined and will be used again in the next result. Combining Theorem \ref{thm.lift.exist.unique} and Proposition \ref{prop.equidistribution.contracting} with the results of Guivarc'h--Raugi \cite{guivarch.raugi.ens} and Benoist--Quint \cite{bq.compositio}, we deduce the following.

\begin{proposition}\label{proposition.FKH}
Let $\mu$ be a probability measure on $\GL(V)$ with finite first moment. Denote by $V=F_1 \supset \cdots \supset  F_k\supset F_{k+1}=\{0\}$ the FKH filtration of $\mu$.
\begin{enumerate}
\item  Let $i \in \{1,\ldots,k\}$ and suppose $\Gamma_{\mu}$ acts irreducibly on $F_i/F_{i+1}$. Let $\nu$ be a $\mu$-stationary ergodic probability measure on $P(F_i)\setminus P(F_{i+1})$. Then the semigroup $\Gamma_\mu$ acts minimally on $\supp(\nu) \setminus P(F_{i+1})$.
\item Moreover, when $\Gamma_{\mu}$ acts irreducibly on each $F_j/F_{j+1}$ for $j=1, \ldots,k$, the map $\nu \mapsto \supp(\nu)\setminus P(F_{i_{\nu}+1})$ yields a bijection between the sets
$$\mathcal{M}_\mu(V) \longleftrightarrow \bigcup_{i=0}^k\{\Gamma_\mu \textrm{-minimal subsets of}\; P(F_i)\setminus P(F_{i+1})\}.$$ 
\end{enumerate}
\end{proposition}

Here, a semigroup $\Gamma$ is said to act minimally on a (non-necessarily compact) topological space $X$ if every orbit is dense or equivalently there is no $\Gamma$-invariant closed proper subspace of $X$.

\subsubsection{Further comments on previous results and our proofs}\label{subsub.history.intro}
The topic of iterated function systems (IFS) of Lipschitz maps on locally compact metric spaces and the associated Markov processes have a substantial history \cite{barnsley-elton,hutchinson,letac,steinsaltz}, see the survey of Diaconis--Freedman \cite{diaconis-freedman} and the recent work of Kloeckner \cite{kloeckner}. In these works, the contraction (or rather contraction-on-average) assumptions allow to prove the existence and uniqueness assertions in Theorem \ref{thm.lift.exist.unique} at a single step, by proving that trajectory-wise images of $X$ converge to a single point independent of the initial point -- this is also the case in Furstenberg's (unique) stationary measure in the irreducible and proximal case. However, in our setting, the lack of contraction assumptions in the basis $P(V/W)$ rules out the possibility of such an approach. Indeed, such a convergence does not need to hold in our situation. We therefore proceed differently. For the existence, we use tools from Markov chain theory, namely the version of Foster--Lyapunov recurrence criterion as recently worked out by B\'{e}nard--de Saxc\'{e} \cite{benard-desaxce}. This work is well-adapted to our purposes in order to tackle the existence problem in the finite first moment case; stronger forms of recurrent behaviour (e.g.~geometric recurrence) can be shown under stronger moment assumptions (e.g.~ finite exponential moment). For uniqueness, we have to deal with the additional lack of contraction in the basis. Inspired by the discussion in \cite[\S 2.2]{aoun-guivarch}, we adopt a fibred dynamical setting and see the space $P(V) \setminus P(W)$ as a fibred system of affine IFS over the compact $P(V/W)$. Our hypotheses then allow us to get fibrewise contraction which in turn suffices to deduce the uniqueness.

Finally, we briefly comment on a different line of study related to our results. The lift measures $\nu$ appearing in Theorem \ref{thm.lift.exist.unique} live naturally on a non-compact subset $P(V) \setminus P(W)$ of the projective space $P(V)$. Whereas strict contraction assumptions such as Hutchinson's \cite{hutchinson} are known to imply that $\nu$ has compact support, under the contraction-on-average assumption, $\nu$ has typically a non-compact support (see a discussion in \cite[\S 5]{aoun-guivarch}). In the particular case of affine random walks, the quantitative properties of the tail of the stationary measure $\nu$ has also been extensively studied by many authors including Kesten \cite{kesten.renewal}, Goldie \cite{goldie}, Guivarc'h--Le Page \cite{guivarch.page.pareto1, guivarch.page.pareto2}, and recently in the more general setting of IFS's by  Kloeckner \cite{kloeckner} -- see the latter work for more detailed overview in this direction.

\subsection{Expanding case}
Our second result treats the expanding case. More precisely,

\begin{theorem}[Partial expansion]\label{thm.mixed}
Let $\mu$ be a probability measure on $\GL(V)$ with a finite first moment and $W$ a $\mu$-invariant subspace. Let $\overline{\nu}$ be a $\mu$-stationary and ergodic probability measure on $P(V/W)$ such that for some $j=1,\ldots,k$, 
\begin{equation}\label{eq.antidom}
 \beta_{j+1}(W) < \alpha(\overline{\nu})  < \beta_{j}(W).   
\end{equation}
Then, the following are equivalent: 
\begin{itemize}
    \item[(i)] There exists $\mu$-stationary lift $\nu$ of $\overline{\nu}$ on $P(V) \setminus P(W)$. 
\item[(ii)] 
There exists a $\Gamma_\mu$-invariant subspace $W'<V$ such that $W' \cap W=F_{j+1}(W)$ and 
$P((W+W')/W)$ is the subspace generated by the support of $\overline{\nu}$.
\end{itemize}
In this case, the lift is unique, has the same cocycle average as $\overline{\nu}$.
\end{theorem}

In the above statement, for a $\mu$-invariant subspace $W<V$,
we denote by 
$W=F_1(W) \supsetneq F_2(W) \cdots \supsetneq F_k(W) \supsetneq \{0\}$ the FKH filtration associated $\mu$-random product on $\GL(W)$ 
 and by 
$\infty>\beta_1(W)>\beta_2(W)>\ldots > \beta_{\min}(W)>-\infty$ the associated FKH exponents. 

Here is a direct consequence that we single out (see Corollary \ref{corol.mixed.intro} for a more general version).  
\begin{corollary}\label{corol.no.lift}
Suppose $\lambda_1(V/W)<\beta_{\min}(W)$. Then 
the following are equivalent: 
\begin{itemize}
  \item[i.] There exists a $\mu$-stationary probability measure on $P(V)\setminus P(W)$. 
    \item[ii.]  
    There  exists a $\Gamma_{\mu}$-invariant subspace of $V$ in direct sum with $W$.  
\end{itemize}
\end{corollary}

Theorem \ref{thm.mixed} generalizes results of Benoist--Bru\`{e}re \cite{benoist-bruere} (as we do not assume any irreducibility or proximality assumptions) which in turn generalized results of Bougerol--Picard \cite{bougerol-picard} pertaining to the setting of affine random walks. However, the previous works also treat the more delicate ``critical case'' under algebraic assumptions.  We do not treat the critical case in this  paper and plan to cover it in a forthcoming work.


Here is another consequence of Theorem \ref{thm.mixed} that gives a slight refinement of the results of Furstenberg--Kifer \cite{furstenberg-kifer} and Hennion \cite{hennion} and generalizing the related result in \cite{aoun-guivarch}. For a more general version, see Corollary \ref{corol.uchech.muchech}.

\begin{corollary}\label{corol.umu.intro}
Fix a Euclidean structure on $V$ and let $\mu^{t}$ denote the image of $\mu$ by transpose map. Let $V=F_1(\mu^t)>F_2(\mu^t)$ be the second FKH space of $\mu^t$ in $V$ and let $V_{1,\mu}:=F_{2}(\mu^t)^\perp$. Then,
any $\mu$-stationary probability measure $\nu$ with $\alpha(\nu)=\lambda_1(\mu)$ is supported in $P(V_{1,\mu})$.
\end{corollary}

The proof of Theorem \ref{thm.mixed} is carried out in sort of an inductive way. At first, we prove its particular case (Theorem \ref{thm.no.lift}, which corresponds to a ``purely expanding'' case), where the index $j$ in \eqref{eq.antidom} is equal to $k$ (i.e.~ $\beta_{k+1}(W)=-\infty$). The proof of this case makes use of the decomposition of stationary measures via the martingale approach due to Furstenberg \cite{furstenberg.non.commuting} to exploit the transient behaviour of the random walk -- this is somewhat in the spirit of original application of Furstenberg. In the second step, by using the work of Furstenberg--Kifer \cite{furstenberg-kifer} and Hennion \cite{hennion}, we reduce the partial expansion to a case where we have a purely expanding and a contracting part. We then combine Theorem \ref{thm.no.lift} with Theorem \ref{thm.lift.exist.unique} to conclude. 

\subsection{An application to homogeneous dynamics}\label{subsec.homogeneous.intro}
Our results can be recast from the point of view of homogeneous dynamics. In this context, they relate to proving (non)-existence and uniqueness of stationary measures on various algebraic homogeneous spaces of type $G/H$ (where $G$ and $H$ are real algebraic groups) and lifts from quotients thereof. We defer the statement of our main result in this context (Theorem \ref{thm.homogeneous}) to \S \ref{sec.homogeneous}. Here we content with Corollary \ref{corol.Poincare.group} below, which follows from an application of our result (combined with Benoist--Quint \cite{bq.compositio}) in a concrete case.






To state it, let $0 \leq k \leq d$ be two integers and $X_{k,d}$ denote the space of affine $k$-spaces in $\R^d$. As a homogeneous space, it can be realized as $G/H$ where $G=\GL_d(\R) \ltimes \R^d$ and $H=P\ltimes \R^{k}$, with $P$ being a maximal parabolic subgroup in $\GL_d(\R)$ given by the stabilizer of a $k$-space in $\R^d$. In the next statement, we consider $\SL_2(\C)$ as a real linear algebraic subgroup of $\SL_4(\R)$ (preserving a complex structure on $\R^4$) and say that a probability measure $\mu$ on a real linear algebraic group is Zariski-dense if the semigroup $\Gamma_\mu$ generated by its support is.

\begin{corollary}\label{corol.Poincare.group}
Let $\mu$ be a Zariski-dense probability measure on $\SL_2(\C) \ltimes \R^4$ with finite first moment. Then,
\begin{itemize}[leftmargin=1cm]
    \item (Bougerol--Picard \cite{bougerol-picard}) There exists no $\mu$-stationary probability measure on $X_{0,4}$.
    \item There exists no $\mu$-stationary probability measure on $X_{1,4}$.
    \item There exists a unique $\mu$-stationary probability measure on $X_{2,4}$.
    \item There exists a unique $\mu$-stationary probability measure on $X_{3,4}$.
\end{itemize}
\end{corollary}

In \cite{benoist-bruere}, Benoist--Bru\`{e}re settled the case where $\mu$ is Zariski-dense in $G$ (their work is an inspiration for our \S \ref{sec.homogeneous}). However, in the situation of the previous corollary, both irreducibility (for the exterior power) and proximality conditions appearing in the work \cite{benoist-bruere} fail to hold. The lack of such algebraic hypotheses in Theorems \ref{thm.lift.exist.unique} and \ref{thm.mixed} allows us to deal with these difficulties.

\bigskip

This article is organized as follows. In \S \ref{sec.ex.and.app}, we discuss some general examples of applications of our results. Section \ref{sec.prelim} is devoted to some preliminaries of random matrix product theory (mainly results of \cite{furstenberg-kifer,hennion} and some consequences). Section \ref{sec.contracting} is devoted to the contracting case, there we prove Theorem \ref{thm.lift.exist.unique} and related results in this case. In Section \ref{sec.expanding}, we tackle the expanding case and prove Theorem \ref{thm.mixed} as well as its corollaries. Finally, in Section \ref{sec.homogeneous}, we adopt the point of view of homogeneous dynamics, prove Theorem \ref{thm.homogeneous} and deduce Corollary \ref{corol.Poincare.group}.

\section{Examples}\label{sec.ex.and.app}
We give examples of various settings where our results from introduction  applies.



\subsection{Lifts of stationary measures on irreducible representations}\label{subsec.irred.ex}
Extending the pioneering work of Furstenberg \cite{furstenberg.non.commuting}, Guivarc'h--Raugi \cite{guivarch.raugi.ens} and later Benoist--Quint \cite{bq.compositio} gave a clear classification of stationary measures on projective spaces for irreducible (more generally completely reducible) representations. The goal here is to discuss some examples of general situations where each stationary measure on $P(V)/P(W)$ is the unique lift of such a Guivarc'h--Raugi--Benoist--Quint measure on the quotient $P(V/W)$ or where the latter measures do not lift.

The following is the simplest example of a setting beyond irreducibility. We include it to illustrate the consequences. Some of the stated consequences of our results in this example also follow from the recent works of Aoun--Guivarc'h \cite{aoun-guivarch} (contracting case, and purely expanding case  if $W$ is  strongly irreducible and proximal), Benoist--Bru\`{e}re \cite{benoist-bruere} (all cases including, additionally, the critical case if $W$ is strongly irreducible and proximal).
\begin{example}\label{example.ip}(\texttt{Strongly irreducible and proximal quotient})
Let $V$ be a finite-dimensional real vector space, $\mu$ a probability measure on $\GL(V)$ with finite first moment such that $\Gamma_\mu$ preserve a subspace $W<V$ and acts strongly irreducibly and proximally on $V/W$. By Furstenberg \cite{furstenberg.non.commuting}, there exists a unique stationary measure (sometimes called the Furstenberg measure) $\nu_F$ on $V/W$. 
\begin{enumerate}[leftmargin=0.5cm]
\item (\texttt{Contracting case}) Suppose $\lambda_1(V/W)>\lambda_1(W)$. Then, it follows from uniqueness of Furstenberg measure $\nu_F$ and by our Theorem \ref{thm.lift.exist.unique} that there exists a unique $\mu$-stationary probability measure on $P(V)\setminus P(W)$. Moreover, by Proposition \ref{proposition.FKH}, there exists a unique $\Gamma_\mu$-minimal set in $P(V) \setminus P(W)$.
\item (\texttt{Expanding case}) Suppose $\lambda_1(V/W)<\lambda_1(W)$.
\begin{itemize}[leftmargin=0.5cm]
\item (\texttt{Purely expanding case})
if $\beta_{\min}(W)>\lambda_1(V/W)$ --- which happens if
the Furstenberg--Kifer--Hennion subspace of $F_2(W)$ is trivial (this happens in particular if $\Gamma_\mu$ acts irreducibly on $W$) --- then the hypotheses of Theorem \ref{thm.no.lift} are satisfied. It implies that either $W$ has an invariant complementary subspace $W'$ (which is, under the assumptions of this example, necessarily unique, strongly irreducible, proximal and satisfies $W \oplus W'=V$) or there does not exist any $\mu$-stationary probability measure on $P(V)/P(W)$.
\item (\texttt{Mixed=Partially expanding case}) Continuing to exclude the critical case (i.e.~equality of $\lambda_1(V/W)$ with some $\beta_j(W)$), the remaining case is there exists $j \geq 2$ with
\begin{equation}\label{eq.under.these}
\qquad
\beta_{\min}(W)< \cdots < \beta_j(W)< \lambda_1(V/W) < \beta_{j-1}(W)\leq \lambda_1(W).
\end{equation}
Since we are $V/W$ is irreducible, in particular it satisfies $F_2(V/W)=\{0\}$ and hence we are able to apply Theorem \ref{thm.mixed}.
\end{itemize}
    \end{enumerate}
The cases described above (contracting, purely and partially expanding) give a description of $\mu$-stationary ergodic probability measures on $P(V) \setminus P(W)$. Provided that one can also describe those on $P(W)$ (e.g.~ if $W$ is irreducible by using Benoist--Quint \cite{bq.compositio} or if one can re-iterate the above analysis for $W$ based on our results), one gets a full description of probability measures on $P(V)$.
\end{example}
\vspace{0.1cm}

Notice that the cases in the above example are exhaustive from algebraic point of view but not from the point of view of Lyapunov exponents. The (only) reason for this is that we are not yet able to handle the critical case (absence of Lyapunov-domination) without any algebraic assumption.

The next setting is a more general one compared to the previous and discusses lifts of Benoist--Quint--Guivarc'h--Raugi measures on the quotient $P(V/W)$. The exhaustive analysis --- contracting, purely and partially expanding cases --- carried out in the previous example can be, in exactly the same way, carried over for the next example (even the irreducibility assumption on $V/W$ can be relaxed to $F_2(V/W)=\{0\}$). 
Aiming to illustrate certain disparate situations, we only focus on the simpler contracting and purely expanding cases --- up to our knowledge, even these cases are not covered by previous works. 

\begin{example}[\texttt{Irreducible quotient}]\label{example.sl2C}
Let $V$ be a finite-dimensional real vector space, $\mu$ a probability measure on $\GL(V)$ such that $\Gamma_\mu$ preserves a subspace $W<V$ and acts irreducibly $V/W$. If the quotient $V/W$ is $\mu$-proximal, the action is necessarily strongly irreducible and we are back to the previous example. So suppose $V/W$ is not $\mu$-proximal. Note this can happen   when $\dim(V/W)\geq 4$, unless the representation in $V/W$ is compact modulo the center. For example,\\[2pt]
\indent \textbf{A.} (\textit{Compact group, unique stationary}) $V/W\simeq \R^{d}$ and the representation of $\Gamma_\mu$ in $V/W$ lives in $\R^\ast O_d(\R)$ and Zariski-dense in $O_d(\R)$. In this case, $P(V/W)$ has a unique $\mu$-stationary probability measure (the $O_d(\R)$-invariant probability measure). \\[2pt]
\indent \textbf{B.} (\textit{Non-compact group, unique stationary}) $V/W\simeq \R^{4}$ and the Zariski-closure of the image of the representation of $\Gamma_\mu$ in $V/W$ is isomorphic to the real group $G=\SL_{2}(\C)$. Then, there exists a unique $\mu$-stationary probability measure on $P(V/W)$. Indeed, thanks to \cite[Theorem 1.7 \& Remark 1.8]{bq.compositio}, $\mathcal{M}_1(P(V/W))$ is in bijection with the $M$-orbits in the set of fixed points of $AN$ on $P(V/W)$, where $A \simeq \mathbb{G}_m^1(\R)$ a $\R$-split torus subgroup of $G$, $N \simeq \R^2$ maximal unipotent subgroup normalized by $A$ and $M \simeq \SO_2(\R)$ is the centralizer of $A$ in a maximal compact subgroup of $G$. Up to conjugacy, with appropriate choices of $A$ (as diagonals) and $N$ (as contained in upper triangular matrices), a straightforward calculation shows that $X^{AN}$ is the real projective line corresponding to the subspace $\langle e_1,e_2 \rangle$ in $\R^4$. The group $M\simeq \SO_2(\R)$ acts by rotations (hence transitively) on $X^{AN}$ and the claim follows. \\[2pt]
\indent \textbf{C.} (\textit{Non-compact group, uncountably many stationary}) $V/W \simeq \bigwedge^{3}\R^{n+1}$ and the Zariski-closure of the image of the representation of $\Gamma_\mu$ in $V/W$ is the image of $\SO(n,1)$ in $\GL(\bigwedge^3\R^{n+1})$. In this case, there are  uncountably many $\mu$-stationary ergodic probability measure on $P(V/W)$ if $n \geq 5$ (see \cite[Remark 1.9]{bq.compositio}).
\begin{enumerate}[leftmargin=1.1cm]
    \item (\texttt{Contracting case}) Suppose $\lambda_1(V/W)>\lambda_1(W)$. Then, since $V/W$ is irreducible, every stationary measure $\nu$ on $P(V/W)$ has the same cocycle average $\alpha(\nu)=\lambda_1(V/W)$. Therefore, by Theorem \ref{thm.lift.exist.unique} for each stationary measure $\nu$ on $P(V/W)$, there exists precisely one $\mu$-stationary measure that lifts $\nu$ in $P(V) \setminus P(W)$. Moreover, combining Proposition \ref{proposition.FKH} and Benoist--Quint \cite{bq.compositio}, we have the following natural bijections:
    \begin{equation}\label{eq.biject.biject}
        \begin{aligned}
        {}\qquad \, \, \mathcal{M}_1(P(V/W))& \simeq \mathcal{M}_1(P(V)\setminus P(W)) \simeq \{\Gamma_\mu\text{-minimal sets in $P(V/W)$}\}\\
        &\simeq \{\Gamma_\mu\text{-minimal sets in $P(V)\setminus P(W)$}\}.
        \end{aligned}
    \end{equation}
    \item (\texttt{Purely expanding case}) Suppose $\lambda_1(V/W)<\beta_{\min}(W)$. Then the hypotheses of Theorem \ref{thm.no.lift} is satisfied and the same conclusion as the previous example holds: Either $W$ has an invariant complementary subspace $W'$ (which is, under the assumptions of this example, necessarily unique, strongly irreducible and satisfies $W \oplus W'=V$) or there does not exist any $\mu$-stationary probability measure on $P(V)/P(W)$.
        \end{enumerate}
Other situations of irreducible, non-proximal representations on $V/W \simeq \R^{4d}$ arise for instance when the Zariski-closure of the image of the representation of $\Gamma_\mu$ in $V/W$ is isomorphic to $\SL_{2d}(\C)$, $\SL_d(\HH)$, $\Sp_d(\C)$, $\SO_{2d}(\C)$ etc.~(see e.g.~\cite[\S 2.2]{chef-israel}).

Finally, notice that the first bijection in \eqref{eq.biject.biject} can be thought of as a generalization of the existence of unique stationary probability measure for the contracting affine recursion where the $V=\R^{d+1}$ and $W=\R^d$, studied by Brandt \cite{brandt}, Bougerol--Picard \cite{bougerol-picard}, and in a more general setting, among others by Diaconis--Freedman \cite{diaconis-freedman} (see references therein).
\end{example}

\begin{remark}\label{rk.mixed.general}
In the above examples, the fact that the quotient $V/W$ is irreducible is only used, for concreteness, to ensure a classification of stationary measures on $P(V/W)$ (thanks to \cite{bq.compositio,guivarch-raugi}). However, if, without any algebraic assumption on $V/W$, one apriori has such a classification on $P(V/W)$, then clearly Theorems \ref{thm.lift.exist.unique} and \ref{thm.mixed} apply under the respective Lyapunov domination conditions to determine the stationary measures on $P(V)\setminus P(W)$.
\end{remark}

\section{Preliminaries}\label{sec.prelim}
In \S \ref{subsec.preliminaries.FK}, we discuss some preliminary material mainly from the work (and consequences thereof) of Furstenberg--Kifer \cite{furstenberg-kifer} and Hennion \cite{hennion}.

\subsection{The Furstenberg--Kifer--Hennion subspaces}\label{subsec.preliminaries.FK}

Unless otherwise specified, all our random variables will be defined on the Bernoulli space $\mathbf{B}:=\GL_d(\R)^{\N}$. An element of $\mathbf{B}$ will be denoted with $b=(b_i)_{i\in \N}$.   Let $\mu$ be a probability measure on $\GL_d(\R)$. We endow $\mathbf{B}$ with the probability measure $\beta:=\mu^{\otimes \N}$. The left random walk associated to $\mu$ is the sequence of random variables $(L_n)_{n\in \N}$ defined for every $n\in \N$ and $b\in \mathbf{B}$ by $L_n(b):=b_n\cdots b_1$. The right random walk $(R_n)_{n\in \N}$ is defined by $R_n(b)=b_1\cdots b_n$. 
We denote by $\Gamma_{\mu}$ the closed semigroup generated by the support of $\mu$. 

The following result of Furstenberg--Kifer and Hennion underlies our considerations in this article.

\begin{theorem}[Furstenberg--Kifer~\cite{furstenberg-kifer}, Hennion~\cite{hennion}]\label{theorem.furstenberg_kifer}
Let $\mu$ be a probability measure on $\GL_d(\R)$ with finite first moment. Then there exists a partial flag $\R^d= F_1 \supset F_2 \supset \ldots \supset F_k\supset F_{k+1}=\{0\}$ of $\Gamma_\mu$-invariant subspaces and a collection of real numbers $\beta_1(\mu)>\ldots>\beta_k(\mu)=:\beta_{\min}(\mu)$ verifying the following: 
\begin{enumerate}
    \item   For every $v \in F_i \setminus F_{i+1}$, we have for $\beta$-almost every $b\in \mathbf{B}$, 
\begin{align*}
\lim_{n\to\infty} \frac{1}{n}\log \|L_n(b) v\|=\beta_i(\mu).
\end{align*}
\item 
The $\beta_i(\mu)$ are the values of
\begin{align*}
\alpha(\nu):= \int_{\mathbb{P}(\R^d)}\int_{\GL_d(\R)}\log\frac{\|gv\|}{\|v\|} d\mu(g)d\nu(\R v)
\end{align*}
that occur when $\nu$ ranges over $\mu$-ergodic $\mu$-stationary probability measures on the projective space $\mathbb{P}(\R^d)$.
\item For every $\mu$-stationary ergodic probability measure $\nu$ on  $P(V)$, letting $i=i(\nu) \in \{1,\ldots,k\}$ be such that $\alpha(\nu)=\beta_i(\mu)$, the subspace $F_{\nu}$ generated by $\supp(\nu)$ verifies $F_{\nu}\subset F_{i}$. In particular $\lambda_1(F_{\nu})=\alpha(\nu)$. \end{enumerate}
\end{theorem}

\begin{remark}\label{rem.examples.FK}
In the result above, the set of exponents $\{\beta_1(\mu),\ldots,\beta_k(\mu)\}$ is contained in the set of Lyapunov exponents of $\mu$ and $\beta_1(\mu)$ coincides with the top Lyapunov exponent $\lambda_1(\mu)$.
 However, this subset can be much smaller. For instance, 

\begin{enumerate}
\item If $G$ acts irreducibly on $V$, then for every probability measure $\mu$ on $G$ such that $\Gamma_\mu$ is Zariski-dense in $G$, we have $F_2=\{0\}$, or equivalently, $k=1$. 
\item The same situation (i.e.~ $F_2=\{0\}$) can also occur in reducible settings. For instance,  denote by $G:=\Aff_d(\R)$ the affine group of $\R^d$ and embed it in $\GL_{d+1}(\R)$ in the usual way via $(x\mapsto Ax+b)\longmapsto \begin{bmatrix}A& b\\
0& 1\end{bmatrix}$. Denote by $W\simeq \R^d$ the invariant subspace of $\R^{d+1}$ corresponding to the fixed points of the translations of $G$.  Let $\mu$ be a probability measure on $G<\GL_{d+1}(\R)$ such that $\Gamma_{\mu}$ does not fix a point in $\R^d$. If $\lambda_1(W)>0$ (expanding case),  then $F_2=\{0\}$. However, if  $\lambda_1(W)<0$ (contracting  case), then $F_2=W$, $\beta_1(\mu)=V$ and $\beta_2(\mu)=\lambda_1(W)$. The other exponents $\beta_i$'s depend on the projection of $\mu$ to the linear part. As one can see, unlike the irreducible setting, the FKH spaces depend heavily on the measure $\mu$. 
\item More generally, one can easily construct probability measures $\mu_1$ and $\mu_2$ with same support such that, for $\mu_1$,  the FKH exponents comprise all Lyapunov exponents and, for $\mu_2$,   the only FKH exponent is the top Lyapunov exponent of $\mu_2$.
\end{enumerate}
\end{remark}
We single out two useful
consequences of this result. For the first one, we refer to   \cite[Proposition 2.4]{prohaska-sert}, \cite[Lemma 2.6]{aoun-sert-ams} or \cite[Lemma 1.5]{eskin-lindenstrauss} for a similar observation.
\begin{corollary}\label{corol.unif.exp}
For every $\epsilon>0$, there exists $n_0 \in \N$ such that for every non-zero $v \in \R^d$ and $n \geq n_0$, we have
$$
\frac{1}{n}{\int{\log \frac{\|gv\|}{\|v\|} d\mu^{\ast n}(g)}} \geq \beta_{\min}(\mu)-\epsilon.
$$
\end{corollary}
The proof is similar to \cite[Proposition 2.4]{prohaska-sert}, we include it for reader's convenience.
\begin{proof}
We argue by contradiction and suppose there exists $\epsilon_0>0$ a sequence $n_j \in \N$ going to infinity such that for every $j \in \N$ there exists $v_j$ satisfying $$\frac{1}{n_j}{\int{\log \frac{\|gv_j\|}{\|v_j\|} d\mu^{\ast n_j}(g)}} \leq \beta_{\min}(\mu)-\epsilon_0.$$ 
Up to passing to a subsequence of $n_j$, we can suppose that the limit $\lim_{j \to \infty}\frac{1}{n_j}\sum_{k=0}^{n_j-1} \mu^{\ast k} \ast \delta_{\R v_j}$ exists. Denote this limit probability measure on $\P(\R^d)$ by $\tilde{\nu}$. Clearly, $\tilde{\nu}$ is $\mu$-stationary. We have
\begin{equation}
\begin{aligned}
\int \int \log \frac{\|gv\|}{\|v\|} d\mu(g) d\tilde{\nu}(\R v) &= \lim_{j \to \infty} \frac{1}{n_j} \sum_{k=0}^{n_j-1} \int \int \log \frac{\|ghv_j\|}{\|hv_j\|} d\mu(g) d\mu^{\ast k}(h)\\
&=\lim_{j \to \infty} \frac{1}{n_j}{\int{\log \frac{\|gv_j\|}{\|v_j\|} d\mu^{\ast n_j}(g)}} \leq \beta_{\min}(\mu)-\epsilon_0
\end{aligned}
\end{equation}
where we used dominated convergence (thanks to the finite first moment assumption) in the first equality and the additive cocycle property of $(g,\R v) \to \log \frac{\|gv\|}{\|v\|}$ in the second equality. In particular, $\tilde{\nu}$ has a $\mu$-stationary ergodic component $\hat{\nu}$ with 
$$\int \int \log \frac{\|gv\|}{\|v\|} d\mu(g) d\hat{\nu}(\R v) \leq\beta_{\min}(\mu)-\epsilon_0
$$
which contradicts (ii) of Theorem \ref{theorem.furstenberg_kifer}.
\end{proof}

In general, a FKH exponent $\beta_i(V/W)$ of the quotient $V/W$ may not appear as a FKH exponent of the full-space $V$. The following observation says that this does not happen if $W$ is already a FKH space.

\begin{corollary}\label{corollary.FK2}
1. For every $1\leq i\leq  k$,  $F_2(F_i/F_{i+1})=\{0\}$.\\[2pt] 
2. For every $1\leq j\leq k$,   the FKH spaces (resp.~ exponents) of $V/F_{j+1}$ are the $F_i/F_{j+1}$'s (resp.~ $\beta_i(V)$'s) for $i=1,\cdots, j$. In particular, $\beta_{\min}(V/F_{j+1})=\beta_{j}(V)$.
\end{corollary}

\begin{proof}
\indent 1. Without loss generality $i=1$. We argue by contraction. If $F_2(V/F_2)\neq \{0\}$ there would exist a $\Gamma_\mu$-invariant non-trivial subspace $W'$ of $V/F_2$ whose top Lyapunov exponent is $<\lambda_1(V/F_2)=\lambda_1(V)$ (the latter identity follows from  \cite[Lemma 3.6]{furstenberg-kifer}). 
Denote by $\pi:V\to V/F_2$ the 
canonical projection and let  $V':=\pi^{-1}(W')$. The latter is  a $\Gamma_{\mu}$-invariant subspace of $V$ containing strictly $F_2(V)$ and  whose top Lyapunov is also $<\lambda_1$ (the last assertion follows for instance by representing the matrices in $W'$ by upper triangular by block matrices with $F_2(V)$ being one block). This contradicts Theorem \ref{theorem.furstenberg_kifer} (1).\\[3pt] 
\indent 2. The following is clearly a $\Gamma_{\mu}$-filtration of $V/F_{j+1}$: 
$V/F_{j+1}=F_1/F_{j+1}\supset F_2/F_{j+1}\supset \cdots F_{j}/F_{j+1}\supset F_{j+1}/F_{j+1}=\{0\}$. Let $i\in \{1,\cdots, j\}$. Denote for simplicity $W_i=F_{i}/F_{j+1}$. By Theorem \ref{theorem.furstenberg_kifer} (1), its top Lyapunov exponent is
$\beta_i(V)$. Also $\lambda_1(W_i/W_{i+1})=\lambda_1(F_i/F_{i+1})=\beta_i(V)$. 
Since by (1) $F_2(W_i/W_{i+1})=\{0\}$, 
by  writing the matrices in $W_i$ as upper triangular matrices by block with $W_{i+1}\subset W_i$ as one invariant block, and using   Theorem \ref{thm.no.lift} (1) in $W_i/W_{i+1}$, we see that the growth rate of $\|L_n v\|$ is $\lambda_1(W_i)=\beta_{i}(V)$ if $v\in W_i\setminus W_{i+1}$.  This being true for every $i=1,\cdots, j$, this  finishes the proof. 
\end{proof}

\section{Contracting case}\label{sec.contracting}

\subsection{The bundle structure behind Theorem \ref{thm.lift.exist.unique}}\label{subsec.bundle.structure}
Let $V$ be a finite dimensional real vector space and $W<V$ a subspace. Let $P_W<\GL(V)$ be the parabolic subgroup of $\GL(V)$ given by the stabilizer of the subspace $W$. We fix a complement $W'$ of $W$ in $V$, a basis of $V$ adapted to the decomposition $V=W\oplus W'$. Let $d=\dim(V)$ and $r=\dim(W)$. The matrix representation of each element $g$ in $P_W$ is of the form 
$g=\begin{bmatrix}A&B\\
0& C\end{bmatrix}$ with $A\in \GL_r(\R)$, $B\in M_{r,d}(\R)$ and $B\in \GL_{d-r}(\R)$.  The matrix $A$ represents the action of $g$ on $W$ and the matrix $C$ represents its action on $V/W$. In the sequel, whenever $g \in P_W$ is used, we will use this notation $A,B$ and $C$ without specifying the dependence on $g$ for simplicity in notation.

Fix a Euclidean structure on $V$ and endow $V/W$ with the associated Euclidean structure. The open subset $P(V)\setminus P(W)$ of $P(V)$ identifies then with the quotient  space $(S^1(V/W) \times W) / \{\pm 1\}$, where $S^1(V/W)$ denotes the unit sphere of $V/W$ and  $\{\pm 1\}$ acts as the scalar multiplication diagonally on both factors. We can concretely express this identification as follows: let $[\xi] \in P(V) \setminus P(W)$ and choose a representative $\xi \in V$. Let $\xi_W$ and $\xi_{W'}$ be such that $\xi=\xi_W + \xi_{W'}$. The identification then writes as $$[\xi]\mapsto (\frac{\xi_W}{||\xi_W||}, \frac{\xi_{W'}}{||\xi_{W}||})/\{\pm 1\}.$$ In the sequel, we will often work with a lift of the right-hand-side and ignore the quotient by $\{\pm 1\}$. This should not cause confusion.  
An element of $S^1(V/W)\times W$ will be usually denoted   $(\theta, t)$ with $\theta\in S^1(V/W)$ and $t$ for an element of $W$.

The previous identification induces a cocycle $\sigma: P_W \times S^1(V/W) \to \Aff(W)$ expressing the action on the product $S^1(V/W) \times W$. Namely,  
$$
g\cdot (\theta,t)=( g\cdot \theta, \sigma(g,\theta)(t)),$$
where for $g=\begin{bmatrix}A& B\\
0& C\end{bmatrix}$, we have  \begin{equation}\label{eq.cocycle.formula}
g\cdot \theta=\frac{C\theta}{||C\theta||} \quad \textrm{and}\quad
\sigma(g,\theta): t\mapsto 
\frac{At}{|| C\theta||} + \frac{B \cdot \theta}{\|C\theta\|}
.
\end{equation}

\bigskip

The following two subsections are devoted to the proof of Theorem \ref{thm.lift.exist.unique}, where we prove the uniqueness and existence assertions, respectively.

\subsection{Uniqueness} 

\begin{proposition}[Uniqueness]\label{prop.uniqueness}
Let $\mu$ be a probability measure on $\GL(V)$ with a finite first moment and  $W$ a $\mu$-invariant subspace. Then for every $\mu$-stationary ergodic probability measure $\overline{\nu}$ on $P(V/W)$ whose cocycle average satisfies $$\alpha(\overline{\nu}) >\lambda_1(W),$$ there exists at most one $\mu$-stationary lift $\nu$ on $P(V)\setminus P(W)$. 
\end{proposition}

For the proof, we will require the following lemma.

\begin{lemma}\label{lemma.convergence}
Keep the assumptions of Proposition \ref{prop.uniqueness} and denote by  $j\in \{1,\cdots, k\}$ the largest integer such that 
$\lambda_1(W)<\beta_j(V/W)$. Let $\theta
\in S^1(V/W)\setminus F_{j+1}(V/W)$. Then,  for any two probability measures $\nu_1$ and $\nu_2$ on $W$ and every function $f \in C_c(W)$, $$
\int_W f(\sigma(L_n, \theta)t) d\nu_1(t) -\int_W f(\sigma(L_n, \theta)t) d\nu_2(t) \underset{n \to \infty}{\overset{L^1(\mathbf{B}, \beta)}{\longrightarrow}} 0.
$$
\end{lemma}

The proof uses the bundle structure \S \ref{subsec.bundle.structure}. Let us first fix our notation. Given an affine space $W$ and an affine map $T$ on $W$, we denote by $L(T)\in \GL(W)$ its linear part.  Here is an immediate property: for every $x,y\in W$: 
\begin{equation}\label{eq.affine.algebra1}
Tx-Ty=\textrm{Lin}(T)(x-y).
\end{equation}

\begin{proof}[Proof of Lemma \ref{lemma.convergence}]
   
By \eqref{eq.cocycle.formula}, the following holds for every $n\in \N$, 
\begin{equation}\label{eq.lin.part.composed}
\textrm{Lin}(\sigma(L_n(b),\theta))=\frac{A(L_n)}{||C(L_n)\theta||}.
\end{equation}
Since $\theta\not\in F_{j+1}(V/W)$,  Theorem \ref{theorem.furstenberg_kifer} shows that, almost surely, \begin{equation}
\label{eq.conva}
\limsup_{n \to \infty} \frac{1}{n} \log \|\textrm{Lin}(\sigma(L_n,\theta))\|\leq  \lambda_1(W)-\beta_j(V/W)<0.
\end{equation}
Thus there exists some $c_0>0$ such that 
\begin{equation}\label{eq.almost.convergence}
\beta\left\{b\in B : \|\textrm{Lin}(\sigma(L_n(b),\theta))\|\geq \exp(-n c_0)\right\}\underset{n\to +\infty}{\longrightarrow} 0.
\end{equation}
Let now $\epsilon>0$. Since $\nu_1$ and $\nu_2$ are probability measures on $W$, we can find some $M>0$ such that $(\nu_1\otimes \nu_2)(\{(t,s)\in W^2\,|\, |t-s|>M\})<\epsilon$. 
We write 
\begin{equation*}
\begin{aligned}
& \left| \int f(\sigma(L_n(b), \theta)t) d\nu_1(t) -\int f(\sigma(L_n(b), \theta)t) d\nu_2(t) \right| = \\
&\left| \iint \left(f(\sigma(L_n(b), \theta)t)   - f(\sigma(L_n(b), \theta)s) \right)d\nu_1(t)d\nu_2(s) \right|  \leq \\
&\leq  \left| \iint_{\{|t-s|
\leq M\}} f(\sigma(L_n(b), \theta)t)-f(\sigma(L_n(b), \theta)s)\,d(\nu_1\otimes \nu_2)(t,s) \right| + 2\|f\|_\infty \epsilon.
\end{aligned}
\end{equation*}
Thus, letting $$I_n:=\int_{B} d\beta(b)\left|
\int f(\sigma( L_n(b), \theta)t) d\nu_1(t) -\int f(\sigma( L_n(b), \theta)t) d\nu_2(t) \right|,$$
we get from the triangular inequality and Fubini's theorem, 

$$ I_n\leq \iint_{\{|t-s|\leq M\}}{\int_{B} d\beta(b)\left|
 f(\sigma(L_n(b), \theta)t)-f(\sigma(L_n(b), \theta)s)
 \right| d(\nu_1\otimes \nu_2)(t,s)}+2||f||_{\infty}\epsilon.
$$
 By \eqref{eq.almost.convergence}, we can find $n_0(\epsilon)\in \N$ such that for every $n\geq n_0(\epsilon)$, 
$$\beta\left\{b\in B : \|\textrm{Lin}
(\sigma(L_n(b),\theta))\|\geq \exp(-n c_0) \right\} <\epsilon,$$
By uniform continuity of $f$, there is some $\delta(\epsilon)>0$ such that $|f(x)-f(y)|<\epsilon$, whenever $|x-y|<\delta(\epsilon)$. Without loss of generality we can assume $\exp(-n_0 c_0)<\delta(\epsilon)/M$ so denoting
  $B'_{n,\epsilon}:=\{b\in B :  \|\textrm{Lin}(\sigma(L_n(b),\theta))\|<\delta(\epsilon)/M\}$, we have for every $n\geq n_0(\epsilon)$,   $$\beta(B'_{n,\epsilon})>1-\epsilon.$$
 Hence  for every $n\geq n_0(\epsilon)$, 
$$ I_n\leq \iint_{\{|t-s|\leq M\}}{\int_{B'_{n,\epsilon}} d\beta(b)\left|
 f(\sigma(L_n(b), \theta)t)-f(\sigma(L_n(b), \theta)s)
 \right| d(\nu_1\otimes \nu_2)(t,s)}+4||f||_{\infty}\epsilon.
$$
By definition of $B'_{n,\epsilon}$, we get from \eqref{eq.affine.algebra1} that for every $n\geq n_0(\epsilon)$,  $b\in B'_{n,\epsilon}$ and $(t,s)\in W^2$ such that $|t-s|\leq M$, 
$$\|\sigma(L_n(b),\theta)t - \sigma(L_n(b),\theta)s\|<\delta(\epsilon) .$$
By definition of $\delta(\epsilon)$, we deduce finally that for every $n\geq n_0(\epsilon)$,  
$$I_n\leq \epsilon+4||f||_{\infty}\epsilon.$$
\end{proof}

\begin{proof}[Proof of Proposition \ref{prop.uniqueness}]
In all the proof, we fix a $\mu$-stationary probability measure $\overline{\nu}$ on  $P(V/W)$ such that $\alpha(\overline{\nu})>\lambda_1(W)$.  Let $j\in \{1,\cdots, k\}$ be such that $\beta_j(V/W)=\alpha(\nu)$. Note that by Theorem \ref{theorem.furstenberg_kifer}, $\overline{\nu}(F_{j+1}(V/W))=0$. 
Each stationary probability measure $\eta$ on $X:=(S^1(V/W)\times W)/\{\pm 1\}\simeq P(V)\setminus P(W)$ has a (unique) $\mu$-stationary lift 
$\tilde{\eta}$ on the product space $\tilde{X}:=S^1(V/W)\times W$ which is   invariant under the involution $(\theta, t)\mapsto (-\theta, -t)$, namely 
$$\int_{\tilde{X}}{g(\theta, t) d\tilde{\eta}(\theta, t):=\int_X{\frac{g(\theta,t)+g(-\theta, -t)}{2} d\eta([
\theta, t])}}.$$
Similarly, $\overline{\nu}$ has a unique $\mu$-stationary and $\pm 1$-invariant lift $\tilde{\overline{\nu}}$ on $S^1(V/W)$. Clearly, for a probability measure $\eta$ on $X$, the projection and lifting operations commute. 

Hence, it enough to show that if 
  $\nu_1$ and $\nu_2$ are two $\mu$-stationary probability measures on the product space $\tilde{X}$ that project to $\tilde{\overline{\nu}}$, then $\nu_1=\nu_2$. 
Let then $\nu_1$ and $\nu_2$ be such probability measures and consider a continuous function  $f: \tilde{X}\to \R$ with compact support.   By stationarity, we have for every $n\in \N$,  
\begin{equation}
\begin{aligned}
&   \int_{\tilde{X}}{f(\theta, t) d\nu_1(\theta, t) -\int_{\tilde{X}} f(\theta, t) d\nu_2(\theta, t) }   =\\
 &   \int_{G} { \left[\int_{\tilde{X}} f(g\cdot (\theta, t)) d\nu_1(\theta, t) - \int_{\tilde{X}} f(g\cdot (\theta, t)) d\nu_2(\theta, t)\right] d\mu^{\ast n}(g) }= \nonumber
\\&   \int_{G} { \int_{S^1(V/W)} d\tilde{ \overline\nu}(\theta) \left[ \int_{W}  f(g\cdot (\theta, t)) 
d\nu_{1,\theta}(t) - \int_{W} f(g\cdot (\theta, t)) d\nu_{2,\theta}(t)\right]  d\mu^{\ast n}(g)} =\nonumber\\
&     \int_{S^1(V/W)}{d\tilde{ \overline\nu}(\theta) \int_{G} d\mu^{\ast n}(g)  \left[\int_{W}{ f(g\cdot (\theta, t)) d\nu_{1,\theta} (t)}  - \int_W{ f(g\cdot (\theta, t)) d\nu_{2,\theta} (t)}   \right] }=\nonumber\\
&  \int_{S^1(V/W)}{d\tilde{ \overline\nu}(\theta) \int_{\mathbf{B}} d\beta(b)  \left[\int_{W}{ f(L_n(b)\cdot (\theta, t)) d\nu_{1,\theta} (t)}  - \int_W{ f(L_n(b)\cdot (\theta, t)) d\nu_{2,\theta} (t)}   \right] }\nonumber
\end{aligned}\end{equation}
The probability measures $\nu_{i,\theta}$ on $W$  for $i=1,2$ and $\theta\in S^1(V/W)$ are the ones given by 
Rokhlin's disintegration theorem.    Fubini's theorem was used  in the fourth line. In the sequel, we assume   that the function $f$ is of the form 
$f((\theta, t))=g(\theta) h(t)$ with $g\in \mathcal{C}_c(S^1(V/W))$ and $h\in \mathcal{C}_c(W)$. 
We then have 
\begin{equation}
\begin{aligned}
&  \left| \int_X{f(\theta, t) d\nu_1(\theta, t) -\int_X f(\theta, t) d\nu_1(\theta, t) }   \right|\leq \nonumber \\
&     \int_{S^1(V/W)}{d\tilde{ \overline\nu}(\theta) \int_{\mathbf{B}} g(L_n(b) \theta) d\beta(b)  \left|\int_{W}{ h(\sigma(L_n(b), \theta)t) d\nu_{1,\theta} (t)}  - \int_W{ h(\sigma(L_n(b), \theta)t) d\nu_{2,\theta} (t)}   \right| } \nonumber\\
& \leq ||g||_{\infty}  \int_{S^1(V/W)}{d\tilde{ \overline\nu}(\theta) \int_{\mathbf{B}}  d\beta(b)  \left| \int_{W}{  h(\sigma(L_n(b), \theta)t) d\nu_{1,\theta} (t)}  - \int_W{ h(\sigma(L_n(b), \theta)t)   d\nu_{2,\theta} (t)}    \right| } 
\end{aligned}
\end{equation}
Applying Lemma \ref{lemma.convergence} for $\tilde{ \overline\nu}$-almost every $\theta\in S^1(V/W)$ and using the    dominated convergence, we obtain that the bound above tends to $0$ as $n \to \infty$ and hence that
$\int_X{f(\theta, t) d\nu_1(\theta, t)}=\int_X{f(\theta, t) d\nu_2(\theta, t)}$. A direct application of the locally-compact version of Arzela--Ascoli's theorem shows that the linear span of functions $f:\tilde{X}\to \R$ of the form $f(\theta, t)=g(\theta)h(t)$ with $g\in \mathcal{C}_c(S^1(V/W))$ and $h\in \mathcal{C}_c(W)$ is dense in $\mathcal{C}_c(X)$. This implies that $\int_{\tilde{X}}{f(\theta, t) d\nu_1(\theta, t)}=\int_{\tilde{X}}{f(\theta, t) d\nu_2(\theta, t)}$ for every $f\in \mathcal{C}_c(\tilde{X})$. Thus $\nu_1=\nu_2$.   
\end{proof}

\subsection{Existence}
We now turn to proving the existence assertion in Theorem \ref{thm.lift.exist.unique}. We state it as a separate statement below.

\begin{proposition}[Existence] \label{prop.existence}
Let $\mu$ be a probability measure on $\GL(V)$ with finite first moment and $W$ a $\mu$-invariant subspace. Then for every $\mu$-stationary ergodic probability measure $\overline{\nu}$ on $P(V/W)$ whose cocycle average satisfies $$\alpha(\overline{\nu}) >\lambda_1(W),$$ there exists a $\mu$-stationary lift $\nu$ on $P(V)\setminus P(W)$. 
\end{proposition}

We first prove the following particular case. The general case (i.e.~ the previous proposition) will be deduced from the particular case by an inductive argument using additionally Proposition \ref{prop.uniqueness}. 

\begin{proposition}\label{prop.existence.simplified}
Let $\mu$ be a probability measure on $\GL(V)$ with finite first moment and $W$ a $\mu$-invariant subspace. Suppose that $\beta_{\min}(V/W)>\lambda_1(W)$. Then, for any $\mu$-stationary ergodic probability measure $\overline{\nu}$ on $P(V/W)$ there exists a $\mu$-stationary lift $\nu$ on $P(V)\setminus P(W)$.
\end{proposition}

The statement is an extension of the results of Benoist--Bru\`{e}re \cite{benoist-bruere} pertaining to the contracting case. The extension concerns both the algebraic assumption (we do not suppose irreducibility) and the moment assumption (we do not assume that the support of $\mu$ is bounded).

In the proof of the above proposition, we will make use of the notation introduced in \S \ref{subsec.bundle.structure}. For simplicity, we will also denote by the same symbol $\overline{\nu}$ a lift of $\overline{\nu}$ to $S^1(V/W)$. This should not cause confusion. We start by a lemma treating a fully contracted case (i.e.~ $\beta_{\min}(V/W)>\lambda_1(W)$) by using a version of Foster--Lyapunov recurrence criterion due to Bénard--de Saxcé \cite{benard-desaxce} that is well-adapted to our purposes.

\begin{lemma}\label{lemma.tight}
Suppose $\beta_{\min}(V/W)>\lambda_1(W)$. Then

\begin{enumerate}
\item for every $x \in X=P(V) \setminus P(W)$, the sequence $(\mu^{\ast n}\ast \delta_x)_{n\in \N}$ is tight in $X$. 

\item  for every $x\in X$,    the sequence $\frac{1}{n}\sum_{i=1}^n{\delta_{L_i(b)\cdot x}}$ is tight for $\beta$-almost every $b\in B$.
\end{enumerate}
\end{lemma}

\begin{proof}
We will first show that there exists $N_0 \in \N$ and a proper continuous  function $f:X \to \R_+$ with the property that for every $\epsilon>0$, there exists $R>0$ such that for every $x \in X$ there exists $n_x \in \N$ satisfying for every $n \geq n_x$ 
\begin{equation}\label{eq.prove.by.BdS}
\mu^{\ast n N_0} \ast \delta_x (f^{-1}([R,\infty)))<\epsilon.    
\end{equation}
Moreover, by continuity of $f$ the constant $n_x$ can be chosen to be bounded as $x$ varies in a compact set of $X$. To show these, let, for $x=(\theta,t)$, the function $f$ be defined by $f(x)=\log (\|t\|+1)$. This is a proper function on $X$. The inequality \eqref{eq.prove.by.BdS} will follow from \cite[Theorem D]{benard-desaxce} if we can show that the (SD) condition in \cite[\S 2.1]{benard-desaxce} is satisfied. To prove the (SD) condition, we check that the conditions of the (SD) criterion in \cite[Lemma 2.2]{benard-desaxce} is satisfied.
We start by showing the following:\\
\textit{Claim} (A): There exist $\lambda>0$ and $R_0>0$ such 
that for all $\alpha>0$, there exists $n_0\in \N$ for every $n\geq n_0$ and $x=(\theta,t)\in X$ with $\|t\|>\exp(n R_0)$,  \begin{equation}\label{eq.claimA}\P(f(L_n\cdot x) - f(x)<-n\lambda)>1-\alpha.\end{equation}
To prove the claim, given $\alpha>0$, let $n_0 \in \N$ be such that for every $n \geq n_0$ and $(\theta,x) \in X$, we have
\begin{equation}\label{eq.alpha.autour.lyap}
\P(\log \|A(L_n)\| \leq (\lambda_1(W)+\alpha)n)>1-\alpha/3
\end{equation}
and
\begin{equation}\label{eq.alpha.autour.betamin}
\P(\log \|C(L_n)\theta\| \geq (\beta_{\min}(V/W)-\alpha)n)>1-\alpha/3
\end{equation}
That such $n_0 \in \N$ exists, follows from Furstenberg--Kesten \cite{furstenberg-kesten} for the inequality \eqref{eq.alpha.autour.lyap} and Furstenberg--Kifer \cite{furstenberg-kifer} (Theorem \ref{theorem.furstenberg_kifer}) for \eqref{eq.alpha.autour.betamin}.
Moreover, since $\log \|B(L_n)\| \leq \log \|L_n\|$, again by Furstenberg--Kesten \cite{furstenberg-kesten}, we can suppose that for every $n \geq n_0$,
\begin{equation}\label{eq.B.lambda}
    \P(\log \|B(L_n)\| \leq (\lambda_1(V)+\alpha)n)>1-\alpha/3
\end{equation}

Putting \eqref{eq.alpha.autour.lyap}, \eqref{eq.alpha.autour.betamin}, and \eqref{eq.B.lambda} together, we deduce that for every $x=(\theta,t) \in X$ and $n \geq n_0$ on an event of probability at least $1-\alpha$, we have
\begin{equation}\label{equation.estimate.difference}
    \begin{aligned}
    &f(L_n \cdot x)-f(x)=\log
    \frac{ \frac{\|A(L_n)t+B(L_n)\theta\|}{\|C(L_n)\theta\|}+1}{\|t\|+1}\\
  &\leq \log \left( \frac{\|A(L_n)t\|}{\|C(L_n)\theta\| \|t\|} + \frac{\|B(L_n)\theta\|}{\|C(L_n)\theta\| (\|t\|+1)} + \frac{1}{\|t\|+1}
  \right)\\  &\leq \log \left(\exp\left(n(\lambda_1(W)-\beta_{\min}(V/W)+2\alpha)\right)+ \frac{\exp \left(n(\lambda_1(V)-\beta_{\min}(V/W)+2\alpha )\right)}{\|t\|} + \frac{1}{\|t\|}  \right)
  \end{aligned}
\end{equation}
Therefore, the claim (A) is easily seen to follow for any choice of $\lambda \in (0,\beta_{\min}(V/W)-\lambda_1(W))$ and $R_0>\lambda_1(V)-\lambda_1(W)$.\\

\textit{Claim} (B): There exists a sequence of integrable random variables $Z_n$ such that $Z_n/n$ converge in $L^1$ and such that the following holds almost surely: 
\begin{equation*}\forall x\in X, \, \, f(L_n \cdot x)-f(x) \leq Z_n\end{equation*}
Using $\frac{1}{\|g\|}\leq \|g\|^{-1}$ for every $g\in \GL_d(\R)$ and the fact that $N(g)\geq 1$, where, we recall $N(g)=\max\{ \|g\|, \|g^{-1}\|\}$, it follows from the middle estimate in \eqref{equation.estimate.difference} that the following almost sure inequality holds for every $x\in X$, $f(L_n\cdot x)-f(x)\leq  Z_n $, with $$Z_n:=\log 3+ 2\log N(L_n).$$
Now since $\mu$ has a first moment, it follows that each $Z_n$ is integrable and, by  Kingman's theorem, $Z_n/n$ converges almost surely and in $L^1$ to a real number. This proves Claim (B). 
 
Claims (A) and (B) imply  that   all conditions of \cite[Lemma 2.2]{benard-desaxce} are satisfied for the random walk on $X$ induced by  $\mu^{\ast N_0}$ for a certain $N_0$. Indeed, let $\lambda$ and $R_0$ be the constants given by Claim (A). Since the variables $Z_n$ given by Claim (B) converge  in $L^1$, there exist $\alpha>0$ and $n_1 \in \N$ such that for every $n\geq n_1$,  $\E(Z'_{n} \mathds{1}_{[0,\alpha]})\leq n\lambda (1-\alpha)$, where $Z'_{n_1}$ is the standard realisation of $Z_{n_1}$ in the sense of \cite{benard-desaxce}.
We can now apply Claim (A) with this $\alpha>0$, which yields some $n_0\in \N$ satisfying \eqref{eq.claimA}. Letting $N_0:=\max\{n_0,n_1\}$, for every $n\geq N_0$, conditions (1),(2) and (3) are satisfied for the Markov chain $(L_{k n}\cdot x)_{k\in \N}$ (with constants $n\lambda$, $n R_0$ and random variable $Z_{n}$) proving \eqref{eq.prove.by.BdS}.

Having established \eqref{eq.prove.by.BdS}, the rest follows from a standard argument. Namely, let $x\in X$ and  let $K$ be a compact subset of $X$ such that $\mu^{\ast k} \ast \delta_x (K)>1-\epsilon$ for every $k=0, \ldots, N_0-1$. Let $n_K=\max\{n_x : x \in K\}$, where $n_x$ is chosen minimally so as to satisfy \eqref{eq.prove.by.BdS}. Since $K$ is compact, $n_K$ is finite. Now, for every $n \geq N_0 N_K$, writing $n=\ell N_0 +k$ with $k \in \{0,\cdots, N_0-1\}$, we have $\mu^{\ast n}\ast \delta_x(f^{-1}([0,R)))= \mu^{\ast \ell N_0} \ast \mu^{\ast k}  \ast \delta_x(f^{-1}([0,R))) \geq (1-\epsilon)^2 \geq 1-2\epsilon$ proving  part (1) of the lemma.

Finally, to prove part (2) of the lemma, observe that the property (SD) that we have established above for the random walk driven by $\mu^{\ast N_0}$ implies, thanks to \cite[Theorem D (ii)]{benard-desaxce}, that for every $\epsilon>0$, there exists $R>0$ such that for every $x \in X$, almost surely $\frac{1}{n}\#\{k \in \{1,\ldots,n\}: f(L_{kN_0}\cdot x) - f(x) \geq R \}\leq \epsilon$. A standard argument now shows that the same holds for the $\mu$-random walk (instead of $\mu^{\ast N_0}$), see e.g.~\cite[Proposition 3.3.(ii)]{benard-desaxce}.
\end{proof}

\begin{proof}[Proof of Proposition \ref{prop.existence.simplified}]
Let $E \subseteq S^1(V/W)$ be the set of generic points of $\overline{\nu}$, i.e.
$$
E=\{\theta \in S^1(V/W) : \frac{1}{n} \sum_{k=1}^n \mu^{\ast k} \ast \delta_\theta \underset{n \to \infty}{\to} \overline{\nu}\}.
$$
By Chacon--Ornstein ergodic theorem, we have $\overline{\nu}(E)=1$ and let $\theta_0\in E$. Using Lemma \ref{lemma.tight} and, in it, taking $x=(\theta_0,t)$ for some $t \in W$, we obtain a $\mu$-stationary probability measure $\nu'$ on $X$ that projects to $\overline{\nu}$, as desired.
\end{proof}

We can now give the proof of the general existence result.

\begin{proof}[Proof of Proposition \ref{prop.existence}]
Let $\pi: V\to W$ be the canonical projection. 
Let $\overline{\nu}$ be a probability measure on $P(V/W)$ such that $\alpha(\overline{\nu})>\lambda_1(W)$. Let $F_{\overline{\nu}}$ be the subspace of $V/W$ generated by the support of $\overline{\nu}$. By replacing $V$ with $\pi^{-1}(F_{\overline{\nu}})$, we can assume without loss of generality that $\alpha(\overline{\nu})=\lambda_1(V/W)$ (see Theorem \ref{theorem.furstenberg_kifer} (iii)). Let $F_2$ be the first proper FKH subspace of $F_1=V/W$ so that $\overline{\nu}$ gives full mass to the open subset $P(V/W)\setminus P(F_2)$ of $P(V/W)$. Let $W_1=\pi^{-1}(F_2)$. This is  a $\Gamma_{\mu}$-invariant subspace of $V$ that contains $W$ and the quotient vector spaces $W_1/W$ and $(V/W)/F_2$ are  $\Gamma_{\mu}$-equivariantly isomorphic  respectively to $F_2$ and $V/W_1$.
For simplicity of notation, we will also denote by $\pi$ the map $P(V)\setminus P(W)\to P(V/W)$ induced by the projection $\pi: V \to V/W$. Let $\tilde{\pi}$ be the projection    $P(V)\setminus P(W_1)\to P(V/W_1)$.  Let also  $\pi_2$ be the canonical projection $P(V/W)\setminus P(F_2)\to P((V/W)/F_2) \simeq P(V/W_1)$.  
With the latter identification and setting $\pi_{|_{P(V)\setminus P(W_1)}}=\pi_1$, we have $\tilde{\pi}=\pi_2\circ \pi_1$, i.e.~ the following diagram is commutative: 
\begin{center}
\begin{tikzpicture}
\def\a{2.5} \def\b{2}

\path
(-\a,0) node (A) {$P(V)\setminus P(W_1)$}      
(\a,0) node (B) {$P(V/W_1)$}
(0,-\b) node[align=center] (C) {$P(V/W)\setminus P(F_2)$}
(\a+2.26,0) node (D) {$\simeq P((V/W)/F_2)$}
;
\begin{scope}[nodes={midway,scale=.75}]
\draw[->] (A)--(B) node[above]{\Large $\tilde{\pi}$};
\draw[->] (A)--(C.120) node[left]{\Large $\pi_{|_{P(V)\setminus P(W_1)}}=\pi_1$ ${}$ };
\draw[->] (C.60)--(B) node[right]{\Large ${}$ $\pi_2$};
\end{scope}
\end{tikzpicture}   
\end{center}
We denote by $\overline{\overline{\nu}}$ the probability measure $\pi_2 {}_\ast \overline{\nu}$ on $P((V/W)/F_2)\simeq P(V/W_1)$. By $\Gamma_\mu$-equivariance, $\overline{\overline{\nu}}$ is $\mu$-stationary. We aim to apply Proposition \ref{prop.existence.simplified} with $V$ as ambient space, $W_1$ as $\Gamma_{\mu}$-invariant space, $\overline{\overline{\nu}}$ as a $\mu$-stationary probability measure on the quotient. Let us check that, with these choices, the hypotheses of Proposition \ref{prop.existence.simplified} are satisfied. Indeed, by definition of  $F_2$, the only Furstenberg--Kifer--Hennion subspaces of $V/W_1$ are $V/W_1$ and $\{0\}$, i.e.~ $F_2(V/W_1)=\{0\}$. Thus $\beta_{\min}(V/W_1)=\lambda_1(V/W_1)=\lambda_1((V/W)/F_2)=\lambda_1(V/W)$ where the last equality follows from \cite[Lemma 3.6]{furstenberg-kifer}. Using the latter result once more, we have then  $\lambda_1(W_1)=\max\{\lambda_1(W), \lambda_1(W_1/W)\}=\max\{\lambda_1(W), \lambda_1(F_2)\}<\lambda_1(V/W_1)=\beta_{\min}(V/W_1)$. 
Therefore, we can apply Proposition \ref{prop.existence.simplified} and deduce that there exists a $\mu$-stationary probability measure $\nu$ on $P(V)\setminus P(W_1)$ such that $\tilde{\pi}_\ast \nu=\overline{\overline{\nu}}$. It remains to show that $\nu$ is a lift of $\overline{\nu}$, i.e.~ $\pi_1 {}_\ast \nu=\overline{\nu}$. 
Since $\tilde{\pi}=\pi_2\circ \pi_1$ and $\overline{\overline{\nu}}=\pi_2 {}_\ast  \overline{\nu}$, both measures $\pi_1 {}_\ast \nu$ and $\overline{\nu}$ on $P(V/W)\setminus P(F_2)$ are lifts of $\overline{\overline{\nu}}$. Applying the uniqueness result (Proposition \ref{prop.uniqueness})
with $V/W$ as ambient space, $F_2$ as an invariant subspace and $\overline{\overline{\nu}}$ as a stationary measure on the quotient $(V/W)/F_2$, we deduce that $\pi_1\ast \nu=\overline{\nu}$ as desired. Note that we can indeed apply Proposition \ref{prop.uniqueness} since we have on the one hand $\alpha(\overline{\overline{\nu}})=\lambda_1((V/W))/F_2)$ --- this follows from our previous observation that all stationary measures on $(V/W)/F_2$ have the same cocycle average which is necessarily $\lambda_1((V/W)/F_2)$ --- and on the other hand $\lambda_1((V/W))/F_2)>\lambda_1(F_2)$. 
\end{proof}

\subsection{Some consequences}\label{subsec.consequences.sec.3}

\begin{proof}[Proof of Proposition \ref{prop.equidistribution.contracting}]
(1) If the sequence $\eta_n:=\frac{1}{n}\sum_{i=1}^n{\mu^{\ast i} \ast \delta_x}$ converges weakly to $\nu$ then, by continuity of the projection map $\pi: P(V)\setminus P(W)\to P(V/W)$, the sequence $\pi \ast \eta_n=$
$\frac{1}{n}\sum_{i=1}^n{\mu^{\ast i} \ast \delta_{\overline{x}}}$ converges weakly to $\pi \ast \nu=\overline{\nu}$.  Conversely, suppose that $\pi \ast \eta_n$ converges to $\overline{\nu}$. 
Let $j \in \{1,\ldots,k\}$ be such that $\beta_j(V/W)=\alpha(\overline{\nu})$ and denote $W':=\pi^{-1}(F_{j+1}(V/W))>W$. 
Note that since $\overline{\nu}(F_{j+1}(V/W))=0$, $x\not\in W'$. 
By Corollary  \ref{corollary.FK2},  $\beta_{min}((V/W)/F_{j+1}(V/W))=\beta_j(V/W)=\alpha(\overline{\nu})$.  Using the $\Gamma_{\mu}$-equivariant isomorphism $(V/W)/F_{j+1}(V/W)\simeq V/W'$, we deduce that $\beta_{\min}(V/W')=\alpha(\overline{\nu})$. 
On the other hand, since $W'/W\simeq F_{j+1}(V/W)$, we deduce from Lemma \cite[Lemma 3.6]{furstenberg-kifer} that $\lambda_1(W')=\max\{\lambda_1(W), \beta_{j+1}(V/W)\}< \alpha(\overline{\nu})$. Hence $\beta_{\min}(V/W')>\lambda_1(W')$. 
Since $x\not\in P(W')$,   Lemma \ref{lemma.tight} yields the tightness of the 
 sequence $\frac{1}{n}\sum_{i=1}^n{\mu^{\ast i} \ast \delta_{x}}$ in $P(V) \setminus P(W')$. Now consider a limit point $\zeta$ of the sequence $\eta_n$. Let $(n_k)_{k\in \N}$ be an increasing sequence such that $\eta_{n_k}\to \zeta$ weakly. By tightness, $\zeta$ is a probability measure on $P(V)\setminus P(W')\subset P(V)\setminus P(W)$. Since by hypothesis $\pi \ast \eta_n\to \overline{\nu}$ weakly, we deduce from the   continuity of $\pi$ that  $\pi \ast \zeta=\overline{\nu}$. By Proposition \ref{prop.uniqueness}, we deduce that $\zeta$ is the unique lift $\nu$ of $\overline{\nu}$. Thus all limit points  of $\eta_n$ are the same, namely $\nu$.  This concludes the proof.\\
\indent (2) The forward direction is direct thanks to the equivariance of the projection. For the 
backward implication, by the same argument as in part (1), we find $W'>W$ such that $\beta_{\min}(V/W')>\lambda_1(W')$ and such that 
$x\notin P(W')$. Then, (2) of Lemma \ref{lemma.tight} yields the tightness of  the sequence of empirical means $\frac{1}{n}\sum_{i=1}^n 
\delta_{X_i \cdots X_1 \cdot x}$. Hence by Breiman's law of large numbers (see for example \cite[Corollary 2.4]{BQ.book}), any limit point 
$\nu_1$ is a $\mu$-stationary probability measure on $P(V)\setminus P(W)$. Moreover, any such limit 
point has $\overline{\nu}$ as projection on $P(V/W')$. But since there exists a unique lift of 
$\overline{\nu}$, this implies that  $\frac{1}{n}\sum_{i=1}^n \delta_{X_i \cdots X_1 \cdot x}$ converges.
\end{proof}

\begin{proof}[Proof of Corollary \ref{corol.sur.Fi/Fi+1}]
By Theorem \ref{theorem.furstenberg_kifer}, each $\mu$-stationary ergodic probability measure $\nu$ on $P(V)$ lives in some $P(F_i)\setminus P(F_{i+1})$. Also its projection  $\overline{\nu}$  on the quotient   $F_i/F_{i+1}$ satisfies   $\alpha(\overline{\nu})=\lambda_1(F_i)> \lambda_1(F_{i+1})$. The uniqueness assertion in Theorem \ref{thm.lift.exist.unique}  shows then that $\nu$ is  the unique lift of $\overline{\nu}$. 
\end{proof}

\begin{proof}[Proof of Proposition \ref{proposition.FKH}]
\begin{enumerate}[leftmargin=1cm]
\item {Denote by $\overline{\nu}$ the projection of $\nu$ on $P(F_i/F_{i+1})$. It is a $\mu$-stationary ergodic probability measure on $P(F_i/F_{i+1})$} and $\nu$ is its unique lift by Theorem \ref{thm.lift.exist.unique}. Let $x\in \supp(\nu)\setminus P(F_{i+1})$.   Since $\Gamma_{\overline{\mu}}$ acts irreducibly on $F_{i}/F_{i+1}$, we deduce from  
\cite[Theorem 1.5]{bq.compositio} that the sequence of probability measures $\frac{1}{n} \sum_{k=1}^n \mu^{\ast k} \ast \delta_{\overline{x}}$ converges to $\overline{\nu}$. 
By Proposition \ref{prop.equidistribution.contracting}, we get that the sequence $\frac{1}{n} \sum_{k=1}^n \mu^{\ast k} \ast \delta_{x}$ converges to $\nu$. Since this holds for every $x\in \supp(\nu)\setminus P(F_{i+1})$, we get that $\supp(\nu) \setminus P(F_{i+1})$ is $\Gamma_\mu$-minimal in $P(F_i)\setminus P(F_{i+1})$. 
\item The map is well-defined thanks to  (1). Let us show that it is injective. Let $\nu_1$ and $\nu_2$ be two $\mu$-stationary ergodic probability measures with $S:=\supp(\nu_1)=\supp(\nu_2)$. Then, by \cite[Theorem 1.5]{bq.compositio} $\pi(S)$ supports a unique $\mu$-stationary probability measure $\overline{\nu}$. Therefore $\pi_\ast \nu_1=\pi_\ast \nu_2=\overline{\nu}$. Then, the uniqueness assertion in Theorem \ref{thm.lift.exist.unique} implies in turn that $\nu_1=\nu_2$. It remains to prove its surjectivity. Consider a $\Gamma_{\mu}$-minimal set  $S\subset  P(F_i)\setminus P(F_{i+1})$ for some $i$. Without loss of generality, $i$ is minimal.  The projection $\pi(S)$ of $S$ on $P(F_i/F_{i+1})$ is a $\Gamma_{\mu}$-minimal subset of $P(V/W)$. By compactness $\pi(S)$ supports a $\mu$-stationary ergodic probability $\overline{\nu}$. By irreducibility of $F_i/F_{i+1}$, all 
stationary measures on $F_i/F_{i+1}$ have the same cocycle average. Hence $\alpha(\overline{\nu})=\beta_i>\beta_{i+1}=\lambda_1(F_{i+1})$. By Theorem \ref{thm.lift.exist.unique}, there exists a $\mu$-stationary lift $\nu$ on $P(F_i)\setminus P(F_{i+1})$. Let $x \in S$. By Proposition \ref{prop.equidistribution.contracting} and \cite[Theorem 1.5]{bq.compositio}, $\frac{1}{n} \sum_{k=1}^n \mu^{\ast k} \ast \delta_x$ converges to $\nu$. 
Since $S$ is a $\Gamma_{\mu}$-invariant closed subset of $P(F_i)\setminus P(F_{i+1})$,   Portemanteau theorem (applied in the $P(F_i)\setminus P(F_{i+1})$) insures that $\nu(S)=1$. Thus   $\supp(\nu)\setminus P(F_{i+1})\subset S$. Since by (ii)  $\Gamma_{\mu}$ acts minimally on $S$, we deduce that the latter inclusion is an equality. The surjectivity of the map follows. 
 \end{enumerate}
\end{proof}

\begin{remark}
In Proposition \ref{proposition.FKH},
the support of an ergodic $\mu$-stationary probability measure $\nu$ may not be minimal in $P(V)$. It is minimal if and only if the support of $\nu$ is compact in $P(F_{i_0})\setminus P(F_{i_0+1})$. 
\end{remark}

\begin{remark}[Invariance of cocycle average]\label{rk.contracting.lift.same.average}
In passing, we note that similar to the corresponding statement in Theorem \ref{thm.no.lift} (but perhaps less surprisingly), the unique lift $\nu$ of $\overline{\nu}$ satisfies $\alpha(\nu)=\alpha(\overline{\nu})$. 
Indeed, clearly, $\alpha(\overline{\nu})\leq \alpha(\nu)$. On the other hand, $\nu$ is a probability measure on $P(F_{\nu})\setminus P(W\cap F_{\nu})$ with $\lambda_1(W\cap F_{\nu})\leq \lambda_1(W)$ and  $F_{\nu}/(W\cap F_{\nu}) \simeq \pi(F_{\nu})=F_{\overline{\nu}}$  so that,  by \cite[Lemma 3.6]{furstenberg-kifer}, $\lambda_1(F_{\nu})=\lambda_1(F_{\nu}/(W\cap F_{\nu}))=
\alpha(\overline{\nu})$ and hence $\alpha(\nu)\leq \alpha(\overline{\nu})$. \end{remark}

\section{Expanding case}\label{sec.expanding}
This section is devoted to the proof of Theorem \ref{thm.mixed} and Corollaries \ref{corol.no.lift} $\&$ \ref{corol.umu.intro} (and their more general versions below) from the introduction. As explained in the introduction, we start the proof of Theorem \ref{thm.mixed} by first proving a particular case (except for the moment assumption) covering the purely expanding case, i.e.~ $\alpha(\overline{\nu})<\beta_{\min}(W)$.  

\begin{theorem}[Purely expanding case: lifts only come from invariant subspaces]\label{thm.no.lift}
Let $\mu$ be a probability measure on $\GL(V)$ with a finite first moment and $W$ a $\mu$-invariant subspace. Let $\overline{\nu}$ be a $\mu$-stationary and ergodic probability measure on $P(V/W)$ such that 
\begin{equation}\label{eq.antidom.sec}
\alpha(\overline{\nu}) <\beta_{\min}(W).   
\end{equation}
Then, the following are equivalent: 
\begin{itemize}
    \item[(i)] There exists a $\mu$-stationary lift $\nu$ of $\overline{\nu}$ on $P(V)\setminus P(W)$. 
\item[(ii)] There exists a $\Gamma_{\mu}$-invariant subspace $W'$ of $V$ in direct sum with $W$ such that $P((W'\oplus W)/W)$ is the projective subspace generated by $\overline{\nu}$.
\end{itemize}
In this case, there exists a unique $\mu$-stationary lift $\nu$ of $\overline{\nu}$ on $P(V)\setminus P(W)$ and it satisfies $\alpha(\nu)=\alpha(\overline{\nu})=\lambda_1(W')$.
\end{theorem}

For the proof, we will require the following version of a classical observation of Furstenberg \cite{furstenberg.non.commuting}.

\begin{lemma}\label{lemma.finite.union}
Let $\mu$ be a probability measure on $\GL(V)$ and $\nu$ a $\mu$-stationary  probability measure on $P(V)$. 
There exists finitely many subspaces $W_1,\cdots, W_r$ of $V$ such that $\bigcup_{i=1}^r{W_i}$ is $\Gamma_{\mu}$-invariant,  $\nu(\bigcup_{i=1}^r{W_i})=1$, $\nu(W_i)=\nu(W_j)$ for every $i,j\in \{1,\cdots, r\}$ and each $W_i$ is of minimal dimension (among subspaces of $P(V)$ charged by $\nu$). 

\end{lemma}

\begin{proof}
Let $E:=\{ W \leq V \, | \, \nu([W])>0 \}$. This is a non-empty subset of subspaces of $V$. Let $r:=\min\{\dim(W)\,|\,W\in E\}$, $E':=\{W\in E\,|\dim(W)=r\}$,  $\alpha:=\sup\{\nu([W])\,|\, W\in E'\}$ and $F:=\{W\in E'\,|\, \nu([W])=\alpha\}$. By definition of $r \in \N$, $\nu([W \cap W'])=0$ for every $W\neq W'\in E'$. Thus $1 \geq \nu(\bigcup_{W\in F} W)=\sum_{W\in F}{\nu(W)}=\sum_{W\in F}{\alpha}$. Thus $F$ must be finite. By stationarity, for any $W \in F$, $\alpha=\nu(W)=\int \nu(g^{-1}W) d\mu(g)$ and hence, by maximality of $\alpha$, we deduce that for $\mu$-almost  every $g$, $gW \in F$. This proves the claim. 
\end{proof}

\begin{remark}\label{rem.sum}
In the setting of previous lemma, let $\nu$ be a $\mu$-stationary and ergodic probability measure on $P(V)$ and  $W_1,\cdots, W_r$ the finite subspaces given by that lemma.    Then $W_1+\cdots +W_r=F_{\nu}\subset F_{i(\nu)}$, where $F_{\nu}$ and $F_{i(\nu)}$ are defined in \S \ref{subsec.preliminaries.FK}. Indeed, by minimality, each $W_i$ is contained in $F(\nu)$ and $W_1+ \cdots +W_r$ is a $\Gamma_{\mu}$-invariant subspace of $P(V)$ charged by $\nu$. 
\end{remark}

\textit{Proof of Theorem \ref{thm.no.lift}.}
We will prove that $F_\nu \cap W=\{0\}$, which will show the direction $(i) \Longrightarrow (ii)$ by taking $W':=F_{\nu}$. The other implication is immediate. 
Denote by $\pi: V\to V/W$ the canonical projection and let  $\nu$ be a $\mu$-stationary probability measure on $P(V) \setminus P(W)$ such that $\pi_\ast \nu=\overline{\nu}$.  Denote by $\{W_1,\cdots, W_r\}$ the subspaces defined in Lemma \ref{lemma.finite.union} applied with $\nu$. 
 Necessarily $\pi(F_\nu)=F_{\overline{\nu}}\subset F_{i(\overline{\nu})}$.  
Without loss of generality, we can suppose 
that $V=W+F_\nu$ so that $\lambda_1(V/W)=\lambda_1(\pi(F_\nu))=
\lambda_1(F_{\overline{\nu}})=
\alpha(\overline{\nu})$.  
 
\begin{enumerate}[leftmargin=1cm]
\item First, we eliminate the case  $W\subset F_\nu$ (i.e.~ $V=F_\nu$).  For a contradiction, suppose $W\subset  F_\nu$. Since $\nu([W])=0$, $W$ is necessarily a proper subspace of $F_\nu$.  Up to passing to a subset $\mathbf{B}'$ of $\mathbf{B}$ of $\beta$-full mass, we know by Furstenberg \cite{furstenberg.non.commuting} that for every $b \in \mathbf{B}'$ there exists a probability measure $\nu_b$ on $P(F_\nu)$ such that  $R_n(b) \nu \longrightarrow \nu_b$ weakly.
Let $H_\nu$ be intersection of stabilizers of $W_i$ for $i=1,\ldots,r$. Clearly, $H_\nu$ is a finite index subgroup of the group $G_\mu$ generated by $\Gamma_\mu$ and let $\tau(1) < \tau(2) <\ldots$ be the sequence of hitting times of $H_\nu$, i.e.~ for every $n\in \N$ and $b \in \mathbf{B}'$, $R_{\tau(n)}(b)\in H_{\mu}$. Note that since $H_\nu<G_\mu$ is finite index, for each $n \in \N$, $\tau(n)$ is almost surely finite (and even has finite exponential moment).
Passing to a subsequence, we can then assume that
\begin{equation}\label{eq.to.Pi}
\frac{R_{\tau(n)}(b)}{||R_{\tau(n)}(b)||}\to \Pi \in \Endo(F_\nu).    
\end{equation}
Since $\lambda_1(F_\nu/W)=\alpha(\overline{\nu})<\beta_{min}(W)\leq \lambda_1(W)$, then necessarily  $\im(\Pi)\subset W$ (this follows for example by representing the elements in $\Gamma_{\mu}$ as upper triangular by block matrices $\begin{bmatrix}A&B\\
0& C\end{bmatrix}$ with $A$ representing the action on $W$ and $C$ the action on $F_{\nu}/W$). 
We claim that     \begin{equation}\label{equation.key}\exists i_0 \in \{1,\cdots, r\} \; \; \text{such that} \; \; \nu([\ker(\Pi) \cap W_{i_0}])=0\end{equation}
Indeed,  $W_1+\cdots +W_r=F_\nu$ (see Remark \ref{rem.sum}). Since $\Pi\neq 0$ (as $\|\Pi\|=1$), we conclude that there exists some $i_0$ such that $\Pi_{|_{W_{i_0}}}\neq 0$. Identity \eqref{equation.key} follows then from the minimality of $W_{i_0}$ among the set subspaces with positive $\nu$-mass. Up to reindexing, denote $W_1=W_{i_0}$.
 
Let us finally reach a contradiction from \eqref{equation.key}. Denote by $\nu_{|_{W_1}}$ the restriction measure to $W_{1}$. Since $\nu=\int{\nu_b d\beta(b)}$ (see \cite{furstenberg.non.commuting}), then possibly by replacing $\mathbf{B}'$ by a further subset of full measure, we can assume that $\nu_b(W_{i})>0$ for every $i=1,\cdots, r$ and $b \in \mathbf{B}'$. Let then $(\nu_b)_{|_{W_{1}}}$ denote the 
restriction of $\nu_b$ to $W_{1}$. Since $R_{\tau(n)}(b)$ stabilizes $W_{1}$ for every $n\in \N$, $R_{\tau(n)} \nu_{|_{W_{1}}} \to (\nu_b)_{|_{W_{1}}}$ weakly. On the other hand, using \eqref{eq.to.Pi} and \eqref{equation.key},   one has also that $R_{\tau(n)} \nu_{|_{W_{1}}}\to \Pi \nu_{|_{W_{1}}}$. Hence $\Pi \nu_{|_{W_{1}}}=(\nu_b)_{|_{W_{1}}}$. 
Hence $\supp((\nu_b)_{|_{W_{1}}})\subset W_{1}\cap W$. Thus $\nu_b([W_{1}\cap W])=\nu_b([W_{1}])>0$. Since $\nu=\int{\nu_b d\beta(b)}$, we conclude that $\nu([W\cap W_{1}])>0$, contradicting  $\nu([W])=0$. 
    
\item Let now $V_1:=F_{\nu}$ and $W_1:=F_{\nu}\cap W$.  The stationary measure $\nu$ lives in $P(V_1)\setminus P(W_1)$. Since   $\pi(F_{\nu})=V_1/W_1$ has a canonical $\Gamma_{\mu}$-equivariant embedding in $V/W$, we can identify  $F_{\overline{\nu}}$ with $P(V_1/W_1)$. Moreover, since $W_1$ is a $\Gamma_{\mu}$-invariant subspace of $W$,   $\beta_{\min}(W_1)\geq \beta_{\min}(W)>\alpha(\overline{\nu})$,  unless $W_1=\{0\}$.  Applying Case (i)  to $V_1$, $W_1$, $\nu$ and $\overline{\nu}$ shows that $W_1=\{0\}$ as desired. 
\end{enumerate}
This shows the equivalence between statements (i) and (ii) of the theorem. Now we show the last statement. If $\nu$ is a $\mu$-stationary probability measure on $P(V)\setminus P(W)$ that lifts $\overline{\nu}$, then by (ii), necessarily $\alpha(\nu)=\alpha(\overline{\nu})$ (since $F_{\nu}$ is a $\Gamma_{\mu}$-invariant complement of $W$ and is $\Gamma_{\mu}$-equivariant isomorphic to $F_{\overline{\nu}}$) .
Finally, to see the uniqueness claim, let $\nu'$ be another lift of $\overline{\nu}$ in $P(V) \setminus P(W)$. By the same argument as in the beginning of the proof, we have $\pi(F_{\nu'})=F_{\overline{\nu}}=\pi(F_{\nu})$. Since $F_\nu$ (and similarly $F_{\nu'}$) are in direct sum with $W$, this implies that  $W \oplus F_\nu= W \oplus F_{\nu'}$. This implies $F_\nu=F_{\nu'}$ and hence that $\nu=\nu'$ (as $F_{\nu}$ is $\Gamma_{\mu}$-isomorphic to $F_{\overline{\nu}}$). Indeed, if not, we can find $v' \in F_{\nu'}$ such that $v'=w+v$ with $w \in W \setminus \{0\}$ and $v \in F_\nu$. Since $V=W\oplus F_{\nu}$ is a $\Gamma_{\mu}$-invariant decomposition, we have that almost surely 
$\lim \frac{1}{n}\log ||L_n v'||\geq \lim \frac{1}{n}\log ||L_n w||$. Since $w\neq 0$, we have by definition of $\beta_{\min
}(W)$ (see Theorem \ref{theorem.furstenberg_kifer})  that 
almost surely 
$$\lim \frac{1}{n}\log \|L_n v'\|\geq \beta_{\min}(W).$$
On the other hand, since $v'\in F_{\nu'}$ and $F_{\nu'}$ is
a $\Gamma_{\mu}$-invariant subspace of $V$ with top Lyapunov exponent equal to $\alpha(\nu')=\alpha(\overline{\nu})$, we have 
$$\lim \frac{1}{n}\log \|L_n v'\|\leq \alpha(\overline{\nu}).$$
This contradicts $\alpha(\overline{\nu})<\beta_{\min}(W)$. \hfill \qed
\vspace*{0.2cm}


\noindent\textit{Proof of Theorem \ref{thm.mixed}.} 
Let $\overline{\nu}$ be as in the statement.\\ 
$(ii) \implies (i)$: Suppose that such a $W'$ exists and let $V':=W+W'$.  Then $V'/F_{j+1}(W)=W/F_{j+1}(W)\oplus W'/F_{j+1}(W)$. The $\Gamma_{\mu}$-invariant subspace $W'/F_{j+1}(W)$ is $\Gamma_{\mu}$-equivariantly isomorphic to $V'/W$ and hence $\overline{\nu}$ lifts to a $\mu$-stationary probability measure $\nu_1$ on $W'/F_{j+1}(W)$.  Clearly, $\alpha(\nu_1)=\alpha(\overline{\nu})=\lambda_1(V'/W)$, where the last equality is due to the fact that $P(V'/W)$ is the subspace generated by the support of $\overline{\nu}$. Since $\alpha(\overline{\nu})
>\beta_{j+1}(W)$, Theorem \ref{thm.lift.exist.unique} (contracting case) applied to $W'$ as ambient space, $F_{j+1}(W)$ as invariant subspace and   $\overline{\nu}$ as a $\mu$-stationary probability measure on the quotient, yields a $\mu$-stationary lift of $\nu_1$ on $P(W')\setminus P(F_{j+1}(W))\subset P(V)\setminus P(W)$ and clearly, $\nu_1$ projects to $\overline{\nu}$. 

$(i) \implies (ii)$: Suppose there exists a lift $\nu$ of $\overline{\nu}$ in $P(V) \setminus P(W)$. Let $F_{\overline{\nu}}$ be the subspace of $V/W$ generated by the support of $\overline{\nu}$ and let $V \geq V_r>W$ be its pre-image in $V$. Clearly, $\nu$ is supported in $P(V_r)$ and hence it is a lift of $\overline{\nu}$ to $P(V_r) \setminus P(W)$. Since $\nu$ gives zero mass to $P(W)$ and hence to $P(F_{j+1}(W))$, it projects to a measure $\nu_1$ on $P(V_r/F_{j+1}(W))$ that gives zero mass to $P(W/F_{j+1}(W))$. The push-forward of $\nu_1$ by the natural projection $V_r/F_{j+1}(W) \to V_r/W$ is precisely $\overline{\nu}$. This means that there is a lift $\nu_1$ of $\overline{\nu}$ from the projective space of $V_r/W\simeq (V_r/F_{j+1}(W))/(W/F_{j+1}(W))$ to $P(V_r/F_{j+1}(W)) \setminus P(W/F_{j+1}(W))$. However, by 2.~ of Corollary \ref{corollary.FK2}, $\beta_{\min}(W/F_{j+1}(W))=\beta_{j}(W)>\alpha(\overline{\nu})$ and hence the hypotheses as well as (i) of Theorem \ref{thm.no.lift} is satisfied. This result then implies that there is a $G_\mu$-invariant subspace $W'<V_r$ containing $F_{j+1}(W)$ such that $W'/F_{j+1}(W)$ and $W/F_{j+1}(W)$ are in direct sum and the subspace generated by $\overline{\nu}$ is
$P((V_r/F_{j+1}(W))/(W/F_{j+1}(W)))\simeq P(W+W'/W)$. This completes the proof of $(i) \iff (ii)$.
It remains to prove the additional claims. So suppose, $(i)$ and $(ii)$.  We have the following diagram:
\begin{wrapfigure}{l}{0.65\textwidth}
\begin{tikzpicture}
\def\a{2} \def\b{2}
\path
(0,\a) node (A) {$P(V_r) \setminus P(W)$}      
(0,0) node (B) {$P(V_r/F_{j+1}W)) \setminus P(W/F_{j+1}(W))$}
(0,-\a) node[align=center] (C) {$P(V_r/W)$}
(3.2,\a) node (D) {$\nu$}
(3.2,0) node (E) {$\nu_1$}
(3.2,-\a) node (F) {$\overline{\nu}$}
(5.2,1) node (G) {Theorem \ref{thm.lift.exist.unique}}
(5.2,-1) node (H) {Theorem \ref{thm.no.lift}}
;
\begin{scope}[nodes={midway,scale=.75}]
\draw[->] (A)--(B) node[right]{};
\draw[->] (B)--(C) node[left]{};
\draw[->] (D)--(E) node[left]{};
\draw[->] (E)--(F) node[left]{};
\draw [thick, -latex] (F) to [bend right] (D);
\end{scope}
\end{tikzpicture}   
\end{wrapfigure}
The $\mu$-stationary probability measure $\overline{\nu}$ determines uniquely the stationary measure $\nu_1$ by Theorem \ref{thm.no.lift}, which itself has a unique $\mu$-stationary lift on $P(V_r)\setminus P(W)$ thanks to Theorem \ref{thm.lift.exist.unique}.  The uniqueness claim follows. Finally, $\alpha(\overline{\nu})=
\alpha(\nu_1)=\alpha(\nu)$, where the first equality follows from Theorem \ref{thm.no.lift} and the second one by Remark \ref{rk.contracting.lift.same.average}. \qed

Here is a consequence which is a more general version of Corollary \ref{corol.no.lift} from the introduction. 

\begin{corollary}\label{corol.mixed.intro} Suppose  $\lambda_1(V/W)<\lambda_1(W)$ and that Furstenberg--Kifer--Hennion exponents of $V/W$ are distinct from those of $W$. Then, there exists a $\mu$-stationary probability measure on $P(V)\setminus P(W)$ if and only if there exists a $\Gamma_{\mu}$-invariant subspace $W'$ of $V$ such that  $W\cap W'\subset F_{j+1}(W)$ for some $j=1,\cdots, k$ and $\beta_{j+1}(W)<\lambda_1(W')$.
\end{corollary}

\begin{proof}  It follows from hypotheses that any $\mu$-stationary probability measure $\overline{\nu}$ on $P(V/W)$ satisfies $\alpha(\overline{\nu})<\lambda_1(W)$. The condition is then necessary thanks to Theorem \ref{thm.mixed}. Suppose now that such a $\Gamma_{\mu}$-invariant subspace $W'$ and such a FKH space $F_{j+1}(W)$ exist. Without loss of generality we can assume that $j+1$ is the minimal index such that $\beta_{j+1}(W)<\lambda_1(W')$.  
Let $\overline{\nu}$ be a $\mu$-stationary probability measure on $W'/(W'\cap F_{j+1}(W))$ with top cocycle average. Since $\beta_{j+1}(W)<\lambda_1(W')$, $\alpha(\overline{\nu})=\lambda_1(W')$.   Replacing if necessary $W'$ with the preimage of $F_{\overline{\nu}}$ by the canonical projection  $W'\longrightarrow W'/(W'\cap F_{j+1}(W))$, we can also assume that $F_{\overline{\nu}}=W'/(W'\cap F_{j+1}(W))$. We will check that the $\Gamma_{\mu}$-invariant subspace $W'':=W'+F_{j+1}(W)$ of $V$ satisfies the requirements of Theorem \ref{thm.mixed} (ii). First, since $W\cap W'\subset F_{j+1}(W)\subset W$, we have clearly that $W\cap W''=F_{j+1}$. This yields a  $\Gamma_{\mu}$-equivariance isomorphism $ W'/(W'\cap F_{j+1}(W))\simeq (W+W'')/W < V/W$.  We  can then pushforward $\overline{\nu}$ to a $\mu$-stationary probability measure on $V/W$ whose subspace generated by its support is $P((W+W'')/W)$ (without changing its cocycle average). The   stationary measure we obtain will be denoted also by $\overline{\nu}$ for simplicity. Finally, by minimality of $j+1$, we have  $\beta_{j+1}(W)<\alpha(\overline{\nu})<\beta_j(W)$. 
 Theorem \ref{thm.mixed} gives then a $\mu$-stationary probability measure on $P(V)\setminus P(W)$. 
 \end{proof}

Note that Corollary \ref{corol.no.lift} follows immediately from the previous one since the hypothesis $\lambda_1(V/W)<\beta_{\min}(W)$ forces $j=k$ in which case $F_{j+1}=\{0\}$.

We end this section by showing an equidistribution result for the (unique) lift of $\overline{\nu}$, when it exists. This will be a direct consequence of the similar result shown in the contracting case and the proof done above. 
\begin{corollary}\label{corol.equidistribution.mixed} Keep the same assumption as in Theorem \ref{thm.no.lift}. Assume that $\overline{\nu}$ has a $\mu$-stationary lift on $P(V)\setminus P(W)$. Then for $x\in P(V)\setminus P(W)$, $\frac{1}{n}\sum_{i=1}^n{\mu^i \ast \delta_{x}}\to \nu$ if and only if $\frac{1}{n}\sum_{i=1}^n{\mu^i \ast \delta_{\overline{x}}}\to \overline{\nu}$. 
\end{corollary}

\begin{proof}
We will use the diagram in the proof of Theorem \ref{thm.no.lift}. Since the bottom-right arrow (i.e.~the correspondence between $\nu_1$ and $\overline{\nu}$) is given via a $\Gamma_{\mu}$-equivariant linear isomorphism between $V_r/W$ and the complement of $W/F_{j+1}(W)$ in $V_r/F_{j+1}(W)$, an equidistribution statement of  $\overline{\nu}$ is equivalent to an equidistribution statement for $\nu_1$. For the top-right arrow, we are in the setting of the contracting case. Hence the claim follows from 
Proposition  \ref{prop.equidistribution.contracting}. 
\end{proof}

Here is a consequence of Theorem \ref{thm.mixed} (mixed case) that allows a slight refinement of the results of Furstenberg--Kifer \cite{furstenberg-kifer} and Hennion \cite{hennion} regarding the supports of stationary measures in $P(V)$.

\begin{corollary}\label{corol.uchech.muchech}
Let $\mu$ be a probability measure on $\GL(V)$ with finite first moment. Fix a Euclidean structure on $V$ and let $\mu^{t}$ denote the image of $\mu$ by transpose map. Let $V=F_1(\mu^t)>F_2(\mu^t)>\ldots>F_\ell(\mu^t)$ be the Furstenberg--Kifer filtration of $\mu^t$ with associated exponents $\lambda_1(\mu)=\beta_1(\mu^t)>\beta_2(\mu^t)>\ldots >\beta_\ell(\mu^t)$. Given a FKH exponent $\beta_r(\mu)$, let $r' \geq 1$ be the largest index with $\beta_{r'}(\mu^t) \geq \beta_r(\mu)$ and set $V_{r,\mu}:=F_{r'+1}(\mu^t)^\perp \cap F_r(\mu)$. Then,
\begin{itemize}
\item[(1)] $V_{r,\mu}$ is a non-trivial subspace of $F_r(\mu)$ such that any $\mu$-stationary probability measure $\nu$ with $\alpha(\nu)=\beta_r(\mu)$ is supported in $P(V_{r,\mu})$.
\item[(2)] $V_{r,\mu}$ is alternatively characterized as the minimal (for the inclusion) $\mu$-invariant subspace $W$ of $F_{r}(\mu)$ such that $\lambda_1(F_r(\mu)/W)<\beta_r(\mu)$.
\end{itemize}
\end{corollary}

Since the proof involves juggling between $\mu$ and $\mu^t$ and various invariant spaces, for clarity, given a $\mu$-invariant space $W$, we write $\lambda_{1,\mu}(W)$ for the top Lyapunov exponent of $\mu$ on $W$. Moreover, we use the term $\mu$-\textit{Lyapunov spectrum} of $W$, to describe the set of Lyapunov exponents of $\mu$ on $W$ with multiplicities.

\begin{proof}
We start by two observations. First, for any subspace $W<V$ that is $\mu$ and $\mu^t$ invariant, we have the equality of Lyapunov exponents $\lambda_i(\mu)=\lambda_i(\mu^t)$ for every $i=1,\ldots, \dim W$. Second, for any $\mu$-invariant subspace $W$, $W^\perp$ is $\mu^t$-invariant and the $\mu$-Lyapunov exponents (with multiplicities) appearing in $V/W$ are the same as $\mu^t$-Lyapunov exponents appearing in $W^\perp$ (see \cite[Proposition 1]{hennion} or \cite[Corollary 3.8]{aoun-guivarch}). Let $r$ and $r'$ be as in the statement.\\[2pt]
\indent (1) We first show that $V_{r,\mu}$ is a non-trivial subspace of $F_r(\mu)$. Indeed, if it is trivial, this implies that $V=F_{r'+1}(\mu^t)+F_r(\mu)^\perp$, where the last two are $\mu^t$-invariant subspaces. By the observation above, $\beta_r(\mu)$ is a $\mu^t$-Lyapunov exponent and hence its multiplicity on $V$ must be less than or equal to the sum of its multiplicities in $\mu^t$-Lyapunov spectra of $F_{r'+1}(\mu^t)$ and $F_r(\mu)^\perp$. The maximal $\mu^t$-Lyapunov exponent in $F_{r'+1}(\mu^t)$ is $\beta_{r'+1}(\mu^t)<\beta_r(\mu)$. So all contribution to $\beta_r(\mu)$-multiplicity comes from $\mu^t$-Lyapunov spectrum of $F_r(\mu)^\perp$. On the other hand, by the observation above, the $\mu^t$-Lyapunov spectrum of $F_r(\mu)^\perp$ are the same as $\mu$-Lyapunov spectrum of $V/F_r(\mu)$. But since $F_r(\mu)$ has top $\mu$-Lyapunov $\beta_r(\mu)$, this means that the $\mu$-Lyapunov spectrum of $V/F_r(\mu)$ has one copy of $\beta_r(\mu)$-missing compared to that of $V$ which results in a contradiction, showing the claim.

Let $\nu$ be a $\mu$-stationary ergodic probability measure with $\alpha(\nu)=\beta_r(\mu)$. By ergodicity and since $V_{r,\mu}$ is $\mu$-invariant, $\nu(V_{r,\mu})$ is zero or one. Suppose for a contradiction that it is zero. So we have $\nu(F_r(\mu) \setminus V_{r,\mu})=1$ and hence we can project $\nu$ onto a $\mu$-stationary probability measure $\overline{\nu}$ on $F_r(\mu)/V_{r,\mu}$. But as $F_r(\mu)/V_{r,\mu} \simeq (F_{r'+1}(\mu^t)^\perp + F_r(\mu))/F_{r'+1}(\mu^t)^\perp$ and by the initial observation above, the $\mu$-Lyapunov spectrum of $(F_{r'+1}(\mu^t)^\perp + F_r(\mu))/F_{r'+1}(\mu^t)^\perp$ is the same as $\mu^t$-Lyapunov spectrum of $F_{r'+1}(\mu^t)$, which is bounded above by $\beta_{r'+1}(\mu^t)<\beta_r(\mu)$. Hence, the top $\mu$-Lyapunov exponent $\lambda_{1,\mu}(F_r(\mu)/V_{r,\mu})$ of $F_r(\mu)/V_{r,\mu}$ is strictly less than $\beta_r(\mu)$. Since $\lambda_{1,\mu}(F_r(\mu))=\beta_r(\mu)$, this also implies that $\lambda_{1,\mu}(V_{r,\mu})=\beta_r(\mu). $ Therefore, $\alpha(\overline{\nu})\leq \lambda_{1,\mu}(F_r(\mu)/V_{r,\mu})<\beta_r(\mu)$. So letting $V'=F_r(\mu)$ and $W'=V_{r,\mu}$, and considering $\overline{\nu}$ on $P(V'/W')$, we are in the setting of Theorem \ref{thm.mixed} and this theorem implies that if $\overline{\nu}$ has a lift to $P(V') \setminus P(W')$, the lift has the same cocycle average as $\overline{\nu}$. But $\nu$ is a lift of $\overline{\nu}$ and  $\alpha(\nu)=\beta_r(\mu)>\alpha(\overline{\nu})$, yielding a contradiction and concluding the proof.\\[2pt]
\indent (2) Let $W$ be a proper subspace of $V_{r,\mu}$ satisfying $\lambda_{1,\mu}(F_r(\mu)/W)<\beta_r(\mu)$. By properness, $F_{r'+1}(\mu^t)$ is properly contained in the $\mu^t$-invariant subspace $W^\perp$ and hence $W^\perp$ has $\mu^t$-top Lyapunov exponent $\lambda \geq \beta_{r'}(\mu) \geq \beta_r(\mu)$. Since by the initial observation, the $\mu$-Lyapunov spectrum of $F_r(\mu)/W$ is the same as $\mu^t$-Lyapunov spectrum of $W^\perp$, we get a contradiction to $\lambda_{1,\mu}(F_r(\mu)/W)<\beta_r(\mu)$ and the proof is done.
\end{proof}

\section{Stationary measures on non-reductive algebraic homogeneous spaces}\label{sec.homogeneous}
As mentioned in the introduction (\S \ref{subsec.homogeneous.intro}), our results have direct consequences and reformulations from the point of view of homogeneous dynamics. 
We now discuss this aspect more in detail by proceeding with a case analysis describing (and commenting on) various situations that occur when trying to describe stationary measures on algebraic homogeneous spaces. Our result (Theorem \ref{thm.homogeneous}) pertains to the case \textbf{2.b.~}below.

\subsection{A case analysis}

Let $G$ be (the real points of) a real algebraic group, $U$ its unipotent radical and $L$ a (reductive) Levi factor so that we have a Levi-decomposition $G=L \ltimes U$. Let $\mu$ be a probability measure on $G$. We say that it has a finite first moment if its image in a (equivalently, in any) faithful algebraic representation of $G$ has a finite first moment. We break the analysis of $\mu$-stationary measures on $G/H$ into several cases as follows. 

\bigskip

\indent \textbf{1.} (\texttt{Reductive quotients}) The work of Benoist--Quint \cite{bq.compositio} allows one to give a complete description when $G$ is a reductive group (i.e.~ $U=\{\id\}$) and $\mu$ is a \textit{Zariski-dense} probability measure on $G$. Their results imply that, for such a probability $\mu$, there exists a $\mu$-stationary probability measure on $G/H$ if and only if $H$ is cocompact in $G$. 

\bigskip

\indent \textbf{2.} (\texttt{Non-reductive quotients}) Suppose now that $G$ is not reductive, i.e.~$U$ is non-trivial. Consider an algebraic subgroup $H$ of $G$ and let $L_0$ be its projection to the Levi factor of $G$. In view of the $G$-equivariant projection
    $$
    G/H \to L/L_0
    $$
any $\mu$-stationary probability measure on $G/H$ descends to a $\mu$-stationary probability measure on $L/L_0$, for any probability measure $\mu$ on $G$. We now have two essentially different situations: 
\bigskip

\indent \textbf{2.a.} (\texttt{Levi projection does not contain a maximal split solvable})  Since $L$ is reductive, if the image $\mu_L$ of $\mu$ under the natural projection $G \to L$ is Zariski-dense (e.g.~ if $\Gamma_\mu$ is Zariski dense in $G$), then we are back to the setting of $\textbf{1.}$ on the base $L/L_0$. In particular if $L_0$ is not cocompact in $L$, there does not exist any $\mu$ stationary probability measure on $G/H$ for such a probability $\mu$. Notice that as for \textbf{1.}, this non-existence of stationary measures applies \textit{for any Zariski-dense probability measure} $\mu$ on $G$, a situation which will be in contrast with the following case. 
\bigskip

\indent \textbf{2.b.} (\texttt{Levi projection contains a maximal split solvable}) Suppose finally that $L_0$ contains a maximal $\R$-split solvable subgroup of $L$. In this case, by compactness, for any probability measure $\mu$ on $G$, there always exists a $\mu_L$-stationary probability measure on $L/L_0$, however it is far less clear whether they lift to $\mu$-stationary measures on the total space $G/H$. The particular case when $G=\GL_d(\R) \ltimes \R^d$, $U_0$ is the trivial group and $L_0=\GL_d(\R)$ (i.e.~ the base $L/L_0$ is trivial) comprises the extensively studied area of stationary measures on affine spaces, we refer to the work of Bougerol--Picard \cite{bougerol-picard}. In the latter case, for a Zariski-dense probability $\mu$ on $G$, in contrast with \textbf{1.~}and \textbf{2.a}, the existence of a $\mu$-stationary probability measure on $G/H$ strongly depends on the Lyapunov exponents of $\mu_L$ appearing in the standard representation of $\GL_d(\R)$. The more recent work of Benoist--Bru\`{e}re \cite{benoist-bruere} shows that this feature also exists in a more general case consisting of a concrete class of quotients (namely $G$ as in the affine case above and $H=P \ltimes \R^k$, where $P$ is the stabilizer in $\GL_d(\R)$ of a $k$-dimensional subspace in $\R^d$). As we now discuss, the latter results can be extended and refined under block Lyapunov domination assumptions by using our results.
\bigskip

From now we assume that the unipotent radical  $U$ of $G$ is a vector group (i.e.~ abelian).  In what follows, we identify $U$ with its Lie algebra $\mathfrak{u}$. We consider a subgroup $H$ of $G$ of type $H=L_0 \ltimes U_0$, where $U_0$ is a connected (closed) subgroup of $U$ and $L_0$ is its normalizer in $L$. We suppose that $L_0$ contains a $\R$-split solvable subgroup of $L$ and hence it is co-compact in $L$.

Given a probability measure $\mu$ on $G$, to study $\mu$-stationary probability measures on $G/H$, we will find appropriate representations of $G$. To do so, we let $G$ act on $\mathfrak{u}$ by affinities with the linear part given by the action of $L$ on $\mathfrak{u}$ (which is the restriction of the adjoint representation) and $U$ acting on itself by translation, namely 
$$(l, u)\cdot w:=l\cdot w+ u.$$ 
One readily checks that this gives a morphism from $G$ to the group $\Aff(\mathfrak{u})$. Also, it is easy to see that $H$ is the stabilizer of the affine space $\{0\}+\mathfrak{u}_0$ in $\mathfrak{u}$. We can now linearize this affine action by letting $V':=\mathfrak{u}\oplus \R$ and $G$ act linearly on $V'$ via   
\begin{equation}\label{eq.action.formula}(l,u) \cdot (w,t)=(tu+l\cdot w,t).
\end{equation}
The subgroup $H$ of $G$ is precisely the stabilizer of the subspace $S:=\mathfrak{u}_0\oplus \R$. Finally, letting  $V=\bigwedge^{\dim S} V'$, $W=\bigwedge^{\dim S} \mathfrak{u}$ and $W_0=\bigwedge^{\dim S}S$, we get a representation of $G$ in $\GL(V)$ such that $W<V$ is $G$-invariant and $H$ is the stabilizer of the line $W_0$ which does not lie in the subspace $W$.

Therefore, by considering the resulting continuous $G$-equivariant injection $\psi$ given by
\begin{equation}\label{eq.def.psi}
\begin{aligned}
G/H & \overset{\psi}{\longrightarrow} P(V) \setminus P(W)\\
gH & \longmapsto g W_0,
\end{aligned}
\end{equation}
we obtain the following commutative diagram
\begin{equation}\label{eq.homogeneous.cylinder}
\begin{tikzcd}
G/H \arrow[hookrightarrow]{r}{\psi} \arrow[twoheadrightarrow]{d} & P(V) \setminus P(W) \arrow[twoheadrightarrow]{d} \\%
L/L_0 \arrow[hookrightarrow]{r}{\overline{\psi}} & P(V/W)  \simeq P(\bigwedge^{\dim \mathfrak{u}_0} \mathfrak{u})
\end{tikzcd}
\end{equation}
Here, the vertical arrows are the canonical projections. The map $\overline{\psi}$ is given by $lL_0 \to l \wedge^{\dim \mathfrak{u}_0} \mathfrak{u}_0$, and it is easy to see that it is a $G$-equivariant homeomorphism onto its closed image. The same is true of $\psi$, see Lemma \ref{lemma.proper.bijection} below. 


In short, under the corresponding dynamical assumptions, diagram \eqref{eq.homogeneous.cylinder} allows us to bring our analysis concerning the right-column of \eqref{eq.homogeneous.cylinder} back to an analysis of stationary measures on $G/H$ (and correspondence between stationary measures on $L/L_0$ and on $G/H$). Using the representation \eqref{eq.action.formula} and notation thereof, we express this in the following result.

\begin{theorem}\label{thm.homogeneous}
Let $G$, $L_0<L<G$, $U_0<U$ and $H=L_0 \ltimes U_0$ be as above and suppose that the unipotent radical $U$ is abelian. Let $d$ be the dimension of $U$ and $k$ that of $U_0$. Given a probability measure $\mu$ on $G$ with finite first moment, let $\lambda_1 \geq \ldots \geq \lambda_d$ be the Lyapunov exponents of $\mu_L$ in the adjoint representation on the Lie algebra $\mathfrak{u}$ of $U$.
\begin{enumerate}[leftmargin=1cm]
\item (Contracting case) For any $\mu_L$-stationary ergodic probability measure $\overline{\nu}$ on $L/L_0$ such that $\alpha(\overline{\psi}_\ast \overline{\nu})> \lambda_1+\ldots+\lambda_{k+1}$, there exists a unique $\mu$-stationary probability measure on $G/H$ that lifts $\overline{\nu}$. In particular, if $\lambda_{k+1}<0$, then there exists a $\mu$-stationary probability measure on $G/H$. 
Moreover, if $\beta_{\min}(\bigwedge^k \mathfrak{u})>\lambda_1+\ldots+\lambda_{k+1}$, then the projection $G/H \to L/L_0$ induces a bijection
$$
\mathcal{M}_{\mu}(G/H)\simeq \mathcal{M}_{\mu_L}(L/L_0).
$$
   \item (Expanding and mixed case) 
Let $\overline{\nu}$ be a $\mu_L$-stationary ergodic probability measure on $L/L_0$. Suppose that $\alpha(\overline{\psi}_\ast \overline{\nu}) \leq  \lambda_1+\ldots+\lambda_{k+1}$ and $\alpha(\overline{\psi}_\ast \overline{\nu})$ is distinct from any Furstenberg--Kifer--Hennion exponent of $\Lambda^{k+1}\mathfrak{u}$. Then there exists a $\mu$-stationary lift of $\overline{\nu}$ on $G/H$ if any only if there exists a $\Gamma_\mu$-invariant subspace $W'$ of $V$ such that $W' \cap W$ is $F_r(W) \subsetneq W$, $\lambda_1(W')=\alpha(\overline{\psi} \ast \overline{\nu})$ and $(W'+W)/W$ is the subspace generated by the support of $\overline{\psi}_\ast \overline{\nu}$, where $r \geq 2$ is the smallest index such that $\beta_{r}(W)<\alpha(\overline{\psi}_\ast \overline{\nu})$. In these cases, the lift of $\overline{\nu}$ is unique. 
\end{enumerate}
\end{theorem}

We single out as a corollary the following version of the conclusion (2) above that one can obtain under an algebraic assumption on $\Gamma_\mu$. This version will be useful in the treatment Corollary \ref{corol.Poincare.group} from the introduction.

\begin{corollary}\label{corol.homogeneous}
Assume moreover that the unipotent radical of the Zariski-closure $\overline{\Gamma}_\mu^Z$ of $\Gamma_\mu$ is contained in $U$. Then, the consequence (2) of Theorem \ref{thm.homogeneous} can be replaced with the following:
\begin{enumerate}[leftmargin=1cm]
\setcounter{enumi}{1}
\item (Partial expansion) Let $\overline{\nu}$ be a $\mu_L$-stationary ergodic probability measure on $L/L_0$. Suppose that $\alpha(\overline{\psi}_\ast \overline{\nu}) \leq  \lambda_1+\ldots+\lambda_{k+1}$ and $\alpha(\overline{\psi}_\ast \overline{\nu})$ is distinct from any Furstenberg--Kifer--Hennion exponent of $\Lambda^{k+1}\mathfrak{u}$. Then there exists a $\mu$-stationary lift of $\overline{\nu}$ on $G/H$ if any only if the subspace $R<\bigwedge^k \mathfrak{u}$ generated by the support of $\overline{\psi}_\ast \overline{\nu}$ in $\bigwedge^k \mathfrak{u}$  satisfies the 
following: for any (equivalently there exists a) couple $(S_\mu,t)$ where  $S_\mu$ is a Levi subgroup of $\overline{\Gamma}_\mu^Z$ and $t \in U$ such that $t S_\mu t^{-1} <L$,
for every $(\ell,u) \in t \overline{\Gamma}^Z_\mu t^{-1}$ and $x\in R$, we have $\ell \cdot x \wedge u \in F_r(\bigwedge^{k+1}\mathfrak{u})$ where $r \geq 2$ is the smallest index such that $\beta_{r}(\bigwedge^{k+1} \mathfrak{u})<\alpha(\overline{\psi}_\ast \overline{\nu})$.
Moreover, the lift is unique.
\end{enumerate}
\end{corollary}

This statement has the advantage that the conditions that appear in it only concerns (exterior powers of) the adjoint representation of $G$. Note for example that in the above corollary, up to conjugating $\mu$, $t$ can be taken to be the identity element and if $k^{th}$ and $(k+1)^{th}$ exterior powers of the $L_\mu$-representation $\mathfrak{u}$ are irreducible, then the last condition above is satisfied if any only if $\overline{\Gamma}_\mu^Z$ is a reductive group. So it says that such a $\mu$-stationary probability measure $\overline{\nu}$ on $L/L_0$ can be lifted if and only if the unipotent radical of $\overline{\Gamma}_\mu^Z$ is trivial.

We now proceed to prove Theorem \ref{thm.homogeneous} and its Corollary \ref{corol.homogeneous}. The proof will be guided by the commuting diagram \eqref{eq.homogeneous.cylinder}. We start by proving the following.

\begin{lemma}\label{lemma.proper.bijection}
The map $\psi:G/H \to P(V) \setminus P(W)$  defined in \eqref{eq.def.psi} is closed.
\end{lemma}

\begin{proof} It suffices to show that $\psi$ is proper. 
To do this, let $g_n H$ be a sequence in $G/H$ that escapes any compact in $G/H$. By definition of $\psi$, we need to show that any limit point of $g_nW_0$ in $P(V)$ belong to $P(W)$. Writing $g_n \in L \ltimes U$ as a tuple $(l_n,u_n) \in L \times U$, since any compact in $G$ is contained in a set of type $\{(g,h) \in A\times B : A<L \; \, \text{and} \, \; B<U \, \; \text{are compacts}\}$, the condition on $g_n$ is equivalent to saying that $(l_n,u_n)$ eventually escapes any subset of $G$ of the form $(A\times B)H$ with $A<L$ and $B<U$ compact. By cocompactness of $L_0$ in $L$, this is equivalent to require that for every every compact subsets $A<L$ and $B<U$ and every large enough $n\in \N$, we have  
$u_n \notin B+A\cdot U_0$. Recall also that $V=\bigwedge^{\dim S} V'$ and $W=\bigwedge^{\dim S}\mathfrak{u}$ and $S=\mathfrak{u}_0 \oplus \R$. To show the convergence $g_n \bigwedge^{\dim S}S \to P(W)$ in $P(V)$, it suffices to show that any $w=(u_0,t) \in S$  with $t \neq 0$, any limit point of  $g_n \R w$ in $P(\mathfrak{u}\oplus \R)$ belong to  $P(\mathfrak{u})$. Now using the representation formula \eqref{eq.action.formula}, we get that for such $w \in S$, we have $g_n w=(t u_n+l_n w,t)$. Therefore 
we have to show that $tu_n +l_n w \to \infty$ in $U$. But if the latter does not go to infinity, it is contained in a compact $B<U$ and hence $u_n \in B + l_n U_0$ for  infinitely many  $n\in \N$. Since $U_0$ has cocompact stabilizer in $L$ by assumption, we deduce that $u_n \in B+A U_0$ for some compact $A \in L$ and for infinitely many $n \in \N$, contradicting the assumption $g_n H \to \infty$ in $G/H$.
\end{proof}

\begin{remark}\label{rk.on.the.closure}
The reason why the fact that the image of $\psi$ being closed is relevant to show the existence of stationary probability measures on $G/H$ is related to our approach and it is explained as follows. In the contracting case, we will construct a stationary measure on $G/H$ by using the recurrence of the induced Markov chain on $P(V) \setminus P(W)$. Starting from a point $x:=\psi(gH) \in P(V)$, any limiting stationary measure of the sequence $\frac{1}{n}\sum_{k=0}^{n-1}\mu^{\ast k} \ast \delta_x$ lives in the closure $ \overline{\psi(G/H)}$. Hence if $\psi$ is proper, we can ``pull-back'' the stationary measure on $\overline{\psi(G/H)}=\psi(G/H)$ and obtain a stationary measure on $G/H$.
\end{remark}


\noindent \textit{Proof of Theorem \ref{thm.homogeneous}.}
We let $G$ act on $\mathfrak{u}\oplus \R$ by \eqref{eq.action.formula}. 
Since $\mathfrak{u}$ acts trivially on $W=\bigwedge^{k+1}\mathfrak{u}$ and $V/W$, the Lyapunov/Furstenberg--Kifer--Hennion exponents of $\mu$ on $W$ and $V/W$ are those of $\mu_L$. The top Lyapunov exponent for $W$ is $\lambda_1+\cdots +\lambda_{k+1}$ and for $V/W$ is $\lambda_1+\cdots + \lambda_k$ for the latter (since $V/W \simeq
\bigwedge^k \mathfrak{u}$ as $L$-module). Therefore $\lambda_1(W)-\lambda_1(V/W)=\lambda_{k+1}$. 
\begin{enumerate}[leftmargin=1cm]
\item Since $\alpha(\overline{\nu})>\lambda_1+\cdots +\lambda_{k+1}=\lambda_1(W)$, we are in the setting of Theorem \ref{thm.lift.exist.unique} and Proposition \ref{prop.equidistribution.contracting}. By Chacon--Ornstein Theorem, the set
$$
\{\overline{x} \in L/L_0 : \lim_{n \to \infty} \frac{1}{n} \sum_{k=0}^{n-1}\mu_L^{\ast k} \ast \delta_{\overline{x}}=\overline{\nu} \}
$$
has full $\overline{\nu}$-measure, so let $\overline{x}$ belong to this set. Let $x \in G/H$ be such that its image under the natural projection $G/H \to L/L_0$ is $\overline{x}$. Now, it follows from the choice of $\overline{x}$ and $L$-equivariance of $\overline{\psi}$ that we have 
$$
\lim_{n \to \infty} \frac{1}{n} \sum_{k=0}^{n-1}\mu_L^{\ast k} \ast \delta_{\overline{\psi}(\overline{x})}=\overline{\psi}_\ast \overline{\nu}.
$$
It follows from Proposition \ref{prop.equidistribution.contracting} and Theorem \ref{thm.lift.exist.unique} that the sequence $\frac{1}{n} \sum_{k=0}^{n-1} \mu^{\ast k} \ast \delta_{\psi(x)}$ converges to a $\mu$-stationary probability measure $\nu_1$ which is the unique lift of $\overline{\psi}_\ast \overline{\nu}$. But since the map $\psi$ is a homeomorphism onto its closed image in $P(V) \setminus P(W)$, $\nu_1$ is supported in $\psi(G/H)$ and hence pulls-back to a $\mu$-stationary probability measure $\nu$ on $G/H$ which is hence the unique lift of $\overline{\nu}$ on $L/L_0$. The other conclusions are direct consequences. 
\item 
Let $\overline{\nu}$ be given as in the statement.   Since $\lambda_{k+1}>0$ there exists $r \geq 2$ such that 
$$
\beta_{r}(W) < \alpha(\overline{\psi}_\ast \overline{\nu})<\beta_{r-1}(W).
$$
In other words, the condition \eqref{eq.antidom.sec} in Theorem \ref{thm.mixed} is satisfied for the $\Gamma_\mu$-representation on $V=\bigwedge^{k+1} (\mathfrak{u} \oplus \R)$ and $W=\bigwedge^{k+1}\mathfrak{u}$.  

Now suppose there exists a $\mu$-stationary lift of $\overline{\nu}$ in $G/H$, call it $\nu$. Then $\psi_\ast \nu$ is a lift of $\overline{\psi}_\ast \overline{\nu}$ and hence condition (i) of Theorem \ref{thm.mixed} is satisfied. This result then gives that there exists a $\Gamma_\mu$-invariant subspace $W'$ of $\bigwedge^{k+1}(\mathfrak{u} \oplus \R)$ such that $W' \cap \bigwedge^{k+1} \mathfrak{u}$ is $F_r(\bigwedge^{k+1} \mathfrak{u}) \subsetneq \bigwedge^{k+1} \mathfrak{u}$, $\lambda_1(W')=\alpha(\overline{\psi} \ast \overline{\nu})$ and $(W'+W)/W$ is the subspace generated by the support of $\overline{\psi}_\ast \overline{\nu}$.

Conversely, suppose the $\Gamma_\mu$-invariant subspace $W'<\bigwedge^{k+1}(\mathfrak{u} \oplus \R)$ has the stated properties. Then, the hypotheses and $(ii)$ of Theorem \ref{thm.mixed} are satisfied, and there exists then a $\mu$-stationary lift $\nu$ of $\overline{\nu}$ on $P(V)\setminus P(W)$. By Chacon--Ornstein, we can find $\overline{x} \in \overline{\psi}(L/L_0)$ such that $\frac{1}{n} \sum_{k=0}^{n-1} \mu_L^{\ast k} \ast \delta_{\overline{x}} \to \overline{\psi}_\ast \overline{\nu}$ as $n \to \infty$. So let $x \in \psi(G/H)$ be such that its image under the projection $P(V) \setminus P(W) \to P(V/W)$ is $\overline{x}$. Applying Corollary \ref{corol.equidistribution.mixed}, we get that $\frac{1}{n} \sum_{k=0}^{n-1} \mu^{\ast k} \ast \delta_x$ converges to $\nu$. But since $x \in \psi(G/H)$ and, thanks to Lemma \ref{lemma.proper.bijection}, $\psi(G/H)$ is closed, it follows that $\nu$ is supported in $\psi(G/H)$ which hence yields a $\mu$-stationary probability measure on $G/H$ that projects to $\overline{\nu}$. Finally, the claim about uniqueness follows from the corresponding assertion in Theorem \ref{thm.mixed}.\qed
\end{enumerate}

We will now deduce Corollary \ref{corol.homogeneous} from (2) of Theorem \ref{thm.homogeneous}. We first need a lemma.

\begin{lemma}\label{lemma.lyapunov.levi}
Let $G$ be a real linear algebraic group and $G=L\ltimes U$ a Levi decomposition of $G$. Let  $\rho: G\to \GL(V)$ be an algebraic representation, $\mu$ a probability measure on $G$ and  $\mu_L:=\pi_\ast \mu$ where $\pi: G \to L$ is the canonical projection. Then the $\mu$-Lyapunov spectrum of $V$ is the same as $\mu_L$-Lyapunov spectrum of $V$. 
\end{lemma}

\begin{proof}
We proceed by induction on $\dim V$. If $\dim V=0$, the claim is trivially true. Suppose $\dim V \geq 1$ and let  $W:=V^{\rho(U)}$ be the subspace $\rho(U)$-fixed vectors. By Lie--Kolchin theorem, $\dim W \geq 1$. Since $U$ is normal in $G$, $W$ is $G$-invariant. Therefore, $\rho$ induces a representation $\rho_W:G \to \GL(W)$ and $U$ is in the kernel of $\rho_W$. It follows that $\mu$-Lyapunov spectrum of $W$ is the same as $\mu_L$-Lyapunov spectrum of $W$. Using \cite[Proposition 1]{hennion} (or \cite[Corollary 3.8]{aoun-guivarch}), since $\dim V/W < \dim V$, the claim follows by induction. 
\end{proof}

\begin{proof}[Proof of Corollary \ref{corol.homogeneous}]
We first show the sufficiency ($\implies$) direction. Given $\overline{\nu}$ on $L/L_0$ as in the statement suppose there exists a lift on $G/H$. By the conclusion in (2) of Theorem \ref{thm.homogeneous}, there exists $\overline{\Gamma}^Z_\mu$-invariant subspace $W'$ of $V$ such that $W' \cap W$ is $F_r(W) \subsetneq W$, $\lambda_1(W')=\alpha(\overline{\psi}_\ast \overline{\nu})$ and $(W'+W)/W$ is the subspace generated by the support of $\overline{\psi}_\ast \overline{\nu}$, where $r \geq 2$ is the smallest index such that $\beta_{r}(W)<\alpha(\overline{\psi}_\ast \overline{\nu})$.
Let $S_\mu$ be a Levi subgroup of $\overline{\Gamma}_\mu^Z$ and $t \in U$ any element such that $t S_\mu t^{-1} <L$ (such an element $t \in U$ does exist, see \cite[VIII. Theorem 4.3]{hochschild.book}).
Note that $\pi(\overline{\Gamma}_\mu^t)=\pi(\overline{\Gamma}_{\mu_t}^Z)$, where $\mu_t$ is the probability measure $t \mu t^{-1}$ on $G$ obtained by pushingforward $\mu$ by conjugation by $t$.  We now claim that 
\begin{equation}\label{eq.tW'.contained}
    tW'\subset F_r \oplus (\bigwedge^k\mathfrak{u} \otimes \R).
\end{equation}
Clearly $tW'$ is $\overline{\Gamma}_{\mu_t}^Z$-invariant and hence $S_{\mu_t}$-invariant. But by the additional algebraic hypothesis (i.e.~ that the unipotent radical of $\overline{\Gamma}_\mu^Z$ is contained in $U$), it follows that $\pi(\overline{\Gamma}_{\mu_t}^Z)=S_{\mu_t}$. Thus $tW'$ is $\pi(\overline{\Gamma}_{\mu_t}^Z)$-invariant. 

Now, suppose \eqref{eq.tW'.contained} does not hold. Since $V=W \oplus (\bigwedge^{k}\mathfrak{u} \otimes \R)$, there 
would exist a vector $x\in tW'$ whose projection $x_W$ on $W$ parallel to $(\bigwedge^k \mathfrak{u}  \otimes \R)$  does not belong to $F_r$. 
On the one hand, since $x_W\not\in F_r$, then the exponential growth of $\|L_n x_W\|$ of the $\mu_L=\mu_{t,L}$-random walk is at least $\beta_{r-1}(W)>\alpha(\overline{\psi}\ast \overline{\nu})$ (recall that the FKH exponents and spaces on $W$ are the same for $\mu$ and $\mu_L$ because $U$ acts trivially on $W$). On the other  hand,  since $tW'$ is $\pi(G_{\mu,t})$-invariant, then
the exponential growth of $\|L_n x\|$ is at most  $\lambda_{1,\mu_{t,L}}(W')$.
But, by Lemma \ref{lemma.lyapunov.levi}, we have $\lambda_{1,\mu_{t,L}}(W')=\lambda_{1,\mu_t}(W')$. Since $\lambda_{1,\mu_t}(W')=\lambda_{1,\mu}(W')>\alpha(\overline{\psi}\ast \overline{\nu})$, we deduce that $\|L_n x\|\leq \|L_n x_W\|$ for all large $n$. This contradicts the fact that the projection on $W$ parallel to $\bigwedge^k \mathfrak{u} \otimes \R$ is $L$-equivariant and  Claim \eqref{eq.tW'.contained} follows. 


Now taking $R':=tW'\cap(\bigwedge^k \mathfrak{u} \otimes \R)$, we get that $R'$ is a $S_{\mu}$-invariant subspace of $V$ that is contained in $\bigwedge^{k} \mathfrak{u} \otimes \R$. Since $tW'>F_r$, we have $tW'=F_r \oplus R'$. Since $tW'$ is $G_{\mu_t}$-invariant, we also have that for every $g\in G_{\mu_t}$,  $gR' \subseteq R'+F_r$.
Since $R'\subset 
\bigwedge^k\mathfrak{u} \otimes  \R$, we have $R'=R\otimes \R$ where $R$ is the subspace $(W'+W)/W$ of $\bigwedge^k \mathfrak{u} \simeq V/W$ which is the subspace generated by the support of $\overline{\psi}_\ast \overline{\nu}$. Finally, since $tW'$ is $G_{\mu_t}$-invariant, for any $g=(l,u)\in G_{\mu_t}$ and $x\in R$, we have $g\cdot (x\wedge \xi)\in R'+F_r$. By formula \eqref{eq.action.formula}, $g\cdot (x\wedge \xi)=lx \wedge (\xi + \xi u)=lx \wedge \xi + \xi(lx \wedge u) \in tW'$. Since $tW'$ contains $F_r$ and hence its projection to $\bigwedge^k \mathfrak{u} \otimes \R$, we have  $lx \wedge u \in F_r$, as claimed. This completes the proof of the sufficiency.

We now prove the converse ($\impliedby$) direction. Suppose there exists a Levi factor $S_\mu$ of $\overline{\Gamma}_\mu^Z$ and $t \in U$ such that the subspace $R<\bigwedge^{k} \mathfrak{u}$ generated by the support of $\overline{\psi}_\ast \overline{\nu}$ has the stated properties. Let $R'=R \otimes \R < V$ and $W'=R' + F_r(\bigwedge^{k+1}\mathfrak{u})$. It follows by the stated property of $R$ that $W'$ is $G_{\mu_{t^{-1}}}$-invariant. Now noting that $\mu_{L}=\mu_{t^{-1},L}$, we can apply the converse direction of (2) in Theorem \ref{thm.homogeneous} for the measure $\mu_{t^{-1}}$ and deduce that there exists a $\mu_{t^{-1}}$-stationary lift $\nu_{t^{-1}}$ of $\overline{\nu}$. Then, the measure $\overline{\nu}=t \nu_{t^{-1}}$ is a $\mu$-stationary lift of $\overline{\nu}$, completing the proof.
\end{proof}

\subsection{Some consequences}
To give explicit examples fitting into the setting of Theorem \ref{thm.homogeneous}, notice that one can take any algebraic representation $U$ of a reductive group $L$ and form the algebraic group $G=L \ltimes U$ whose unipotent radical is the vector group $U$. Letting $L_0$ be a parabolic subgroup of $L$ given by the stabilizer of a subspace $U_0<U$ and setting $H=L_0 \ltimes U_0$, we are in the setting of Theorem \ref{thm.homogeneous}. We treat two examples (the second one is Corollary \ref{corol.Poincare.group}).

\begin{example}[Benoist--Bru\`{e}re \cite{benoist-bruere}]
Let $L=\GL_d(\R)$ with its standard action on $U=\R^d$. Given a subspace $U_0<U$, let $L_0<L$ be the stabilizer of $U_0$ in $L$, which, in this case, is a parabolic subgroup. In this case, except for the ``critical'' case $\lambda_k=0$, Corollary \ref{corol.homogeneous} implies \cite[Theorem 1.3]{benoist-bruere}. 
\end{example}

We now proceed to prove Corollary \ref{corol.Poincare.group}. The representations appearing in it are not proximal (neither irreducible) and hence it constitutes a new example that can be treated in dominated cases by our results (and not directly by \cite{aoun-guivarch, benoist-bruere, bougerol-picard}).\\


\noindent \textit{Proof of Corollary \ref{corol.Poincare.group}.}
We first setup the homogeneous setting so as to apply Theorem \ref{thm.homogeneous} and its Corollary \ref{corol.homogeneous}. Let $L=\GL_4(\R)$ act naturally on 
$\mathfrak{u}=\R^4$ and $G:=L\ltimes \mathfrak{u}$. Let $(e_1,\ldots,e_4)$ be the canonical basis of $\mathfrak{u}=\R^4$. 
For convenience in calculations, we realize $\SL_2(\C)$ as a subgroup of $\SL_4(\R)$ preserving the complex structure given by the linear transformation $J=\begin{bmatrix} J_1& 0\\
0 & J_1\end{bmatrix}$ with $J_1=\begin{bmatrix}0& -1\\
1& 0\end{bmatrix}$.
For $k\in \{1,2,3\}$, denote by $\mathfrak{u}_{0,k}$ the subspace $\langle e_1,\cdots, e_{k}\rangle$. Set also $\mathfrak{u}_{0,0}=\{0\}$. Let $L_{0,k}<L$ be the maximal parabolic subgroup given by the stabilizer of $\mathfrak{u}_{0,k}$ in $L$ and set $H_k=L_{0,k}\ltimes \mathfrak{u}_{0,k}$. Then, we have $ X_{k,4}\simeq G/H_k$ and, for each $k=0,\ldots,3$, the assertions in Corollary \ref{corol.Poincare.group} are equivalent to the corresponding ones on $G/H_k$.
Having defined the algebraic groups $G$ and $H$, we now note that for $k=1,\ldots,3$, the map $\overline{\psi}:L/L_0 \to P(\bigwedge^k \R^4)$ appearing in Theorem \ref{thm.homogeneous} is given by $\ell L_0 \mapsto \ell \wedge^k \mathfrak{u}_{0,k}$ --- for $k=0$, $L/L_0$ and $P(\bigwedge^0 \R^4)$ are singletons, so the map $\overline{\psi}$ is trivial. 

Let now $\mu$ be a probability measure on $G$ such that $\Gamma_{\mu}$ is Zariski dense in $\SL_2(\C) \ltimes U<G$ and $\mu_L$ its projection on $\SL_2(\C)$. We now make some observations on the representations of $\SL_2(\C)$ on $\bigwedge^k \R^4$ and the Lyapunov exponents.
\begin{itemize}[leftmargin=1cm]
\item To describe the Lyapunov exponents of $\mu_L$ on $\mathfrak{u}=\R^4$, note that representation of $\SL_2(\C)$ on $\R^4$ is irreducible  with proximality index two and the $\R$-split torus of $\SL_2(\C)$ is one-dimensional. Hence it follows from the work of Goldsheid--Margulis \cite{gold.marg}  (see \cite[Lemma 6.23, Corollary 10.15]{BQ.book}) that the Lyapunov exponents of $\mu_L$ (or equivalently $\mu$) acting on $\R^4$ are of the form $\lambda_1=\lambda>0$, $\lambda_2=\lambda$, $\lambda_3=-\lambda$ and $\lambda_4=-\lambda$. Finally, $F_2(\R^4)=\{0\}$ (since $\SL_2(\C) \acts \R^4$ irreducibly).
\item The representation of $\SL_2(\C)$ on $\bigwedge^2 \R^4$ is reducible. More precisely, $\bigwedge^2 \R^4=E \oplus F$ where 
$E=\langle e_1 \wedge e_2, e_3\wedge e_4, e_1\wedge e_3+e_2\wedge e_4, e_1\wedge e_4-e_2\wedge e_3\rangle$ and $F=\langle e_1\wedge e_4+e_2\wedge e_3, e_1\wedge e_3-e_2\wedge e_4\rangle$, with $\SL_2(\C)$ acting irreducibly on $E$ and acting trivially on $F$. The Lyapunov exponents of $\mu_L$ on $E$ are $2\lambda, 0, 0 -2\lambda$ and those on $F$ are $0,0$. We have $F_2(\bigwedge^2 \R^4)=F$. 
\item The representation of $\SL_2(\C)$ on $\bigwedge^3 \R^4$ is isomorphic as $\SL_2(\C)$-module to the chosen  representation on $\R^4$. 
\item Finally, $\SL_2(\R^4)$ acts trivially on $\bigwedge^4 \R^4$. 
\end{itemize}

We now prove each assertion corresponding to $k=0,1,2,3$.

\begin{enumerate}[leftmargin=1cm]
\item Case $k=0$. The quotient $L_0/L_{0,0}$ is trivial and the (unique) $\mu$-stationary probability measure $\overline{\nu}$ on $L_0/L_{0,0}$ satisfies $\alpha(\overline{\psi}\ast \overline{\nu})=0<\lambda_1=\lambda$. We are then in case (2) of Theorem \ref{thm.homogeneous} and Corollary \ref{corol.homogeneous}. Since $F_2(\R^4)=\{0\}$, by Corollary \ref{corol.homogeneous}, there exists a lift of $\overline{\nu}$ if any only of $\overline{\Gamma}_\mu^Z$ can be conjugated into $\SL_2(\C)$. But since $\overline{\Gamma}_\mu^Z$ has a non-trivial unipotent radical, this is not possible and we are done.

\item Case $k=1$. First, we claim that the quotient $L/L_{0,1}$ has a unique $\mu_L$-stationary probability measure $\overline{\nu}$. Indeed, on the one hand, by compactness, it has at least one $\mu_L$-stationary probability measure. On the other hand, by Example \ref{example.sl2C}.(B) (which relies on the work of Benoist--Quint \cite{bq.compositio}), the range of $\overline{\psi}$ in $P(\R^4)$ has a unique $\mu_L$-stationary probability measure. The claim follows.
Now, we have $\alpha(\overline{\psi}\ast \overline{\nu})=\lambda_1(\R^4)=\lambda< 2 \lambda=\lambda_1+\lambda_{2}$. We are hence again in the setting of Corollary \ref{corol.homogeneous}. 
Note that $F_2(\bigwedge^2 \R^4)=F$ and the projective subspace of $P(\R^4)$ generated by the support of $\overline{\psi}\ast \overline{\nu}$ is all of $P(\R^4)$ because $\SL_2(\C)$ acts irreducibly on $\R^4$.  Thus by Corollary \ref{corol.homogeneous}, it is enough to show that  there exists $u,x\in \mathfrak{u}$  such that $x\wedge u\not\in F$. However, it is clear from the expression of $F$ that for any $x\in \R^4$, $x\wedge e_1\not\in F$. The desired condition is satisfied and there is indeed no $\mu$-stationary lift of $\overline{\nu}$ in $G/H_1$, and hence no stationary measure on $G/H_1$. 

\item Case $k=2$. First, we claim that $L/L_{0,2}$ has a unique $\mu_L$-stationary probability measure $\overline{\nu}$. Again, by compactness, it has at least one such measure. Recall that the subspace $E$ of $\bigwedge^2 \R^4$ is $L$-invariant and irreducible, and that the Lyapunov exponents of $\mu_L$ on $E$ are $2\lambda_1,0,0,-2\lambda_2$. 
Since $\overline{\psi}(\id L_0)=e_1\wedge e_2\in E$, it follows that $\overline{\psi}(L/L_{0,2})\subset P(E)$. Now  $\SL_2(\C)$ acts on $E$ strongly irreducibly and proximally (for instance because the action is irreducible and $\lambda_{1,\mu_L}(E)>\lambda_{2,\mu_L}(E)$). Thus, by Guivarc'h--Raugi's theorem \cite{guivarch-raugi}, there exists a unique $\mu$-stationary probability measure on $P(E)$ and the claim is proved.

We now show that this unique stationary measure has a lift on $G/H$. Since $\SL_2(\C)$ acts irreducibly on $W_3=\bigwedge^3 \R^4$, we have $\beta_{\min}(W_3)=\lambda_1(W_3)=\lambda< 2\lambda=\alpha(\overline{\psi}\ast \overline{\nu})$. Theorem \ref{thm.homogeneous} (1) insures then that there is a unique $\mu$-stationary lift of $\overline{\nu}$ on $G/H_2$ and we are done.

\item Case $k=3$. First, we check that $L/L_{0,3}$ has a unique $\mu$-stationary measure $\overline{\nu}$. Since the $\SL_2(\C)$-representation $\bigwedge^3 \R^4$ is isomorphic to the one on $\R^4$, it follows as for case $k=1$ that there exists a unique $\mu$-stationary probability measure $\overline{\nu}$ on $L/L_{0,3}$. We have $\alpha(\overline{\psi}\ast \overline{\nu})=\lambda>0=\lambda_1(W_3)$. The claim then follows directly from (1) of Theorem \ref{thm.homogeneous}. \qed
\end{enumerate}



 \bibliographystyle{abbrv} 

\end{document}